\documentclass[10pt,reqno,a4paper]{amsart}
\usepackage{euscript}
\usepackage{amssymb}
\usepackage{graphicx}
\usepackage{color}
\usepackage{subfigure,graphicx}
\usepackage{overpic}
\renewcommand{\epsilon}{\varepsilon}
\usepackage{float}
\usepackage{soul}
\usepackage{hyperref}
\graphicspath{ {./figures/} }

\newcommand{\R}{\ensuremath{\mathbb{R}}}

\def\e{\varepsilon}

\newtheorem {theorem} {Theorem} 
\newtheorem {proposition} {Proposition}

\newtheorem {lemma} {Lemma}

\newtheorem {remark} {Remark}

\newtheorem {mtheorem} {Theorem}

\begin{document}
\allowdisplaybreaks
\title[Invariant Tori and Periodic Orbits  in the FitzHugh-Nagumo System]{Invariant Tori and Periodic Orbits  in the FitzHugh-Nagumo System}

\begin{abstract}
The FitzHugh-Nagumo system is a $4$-parameter family of $3$D vector field used for modeling neural excitation and nerve impulse propagation. The origin represents a Hopf-zero equilibrium in the FitzHugh-Nagumo system for two classes of parameters. In this paper, we employ recent techniques in averaging theory to investigate, besides periodic solutions, the bifurcation of invariant tori within the aforementioned families. We provide explicit generic conditions for the existence of these tori and analyze their stability properties.
Furthermore, we employ the backward differentiation formula to solve the stiff differential equations and provide numerical simulations for each of the mentioned results.
\end{abstract}

\author{Murilo R. C\^{a}ndido$^1$, Douglas D. Novaes$^2$, Nasrin Sadri$^3$}

\address{$^1$ Department of Mathematics and Computer Science, Faculty of Science and Technology - São Paulo State University (UNESP), Rua Roberto Simonsen, 305 - Centro Educacional, 19060-900, P. Prudente, São Paulo, Brazil} 
\email{mr.candido@unesp.br}

\address{$^2$ Universidade
Estadual de Campinas (UNICAMP), Departamento de Matem\'{a}tica - Instituto de Matem\'{a}tica, Estat\'{i}stica e Computa\c{c}\~{a}o Cient\'{i}fica (IMECC), Rua S\'{e}rgio Baruque de Holanda, 651, Cidade Universit\'{a}ria Zeferino Vaz, 13083--859, Campinas, SP, Brasil} 
\email{ddnovaes@unicamp.br}

\address{$^3$ School of Mathematics, Statistics and Computer Science, College of Science, University of Tehran, Tehran, Iran and School of Mathematics, Institute for Research in Fundamental Sciences (IPM), P. O. Box 19395-5746, Tehran, Iran} 
\email{n.sadri@ut.ac.ir}




\maketitle

\section{Introduction and Statement of the Main Results}\label{sec1}

FitzHugh \cite{FitzHugh} and Nagumo et. al \cite{Nagumo} introduced the following $4$-parameter family of $3$D vector field as a model to describe the excitation of a nerve axon and the propagation  of the messages carried by neurons along an axon:
\begin{equation}\label{FZ}
\begin{aligned}
\dot{x}&=z,\\
\dot{y}&=b (x- d y),\\
\dot{z}&=(x-a) (x-1) x+y+cz,
\end{aligned}
\end{equation}
where $a,b,c,d \in \mathbb{R}.$

The averaging method has found success in examining the periodic behavior of the FitzHugh-Nagumo system. For instance, in \cite{RodrigoLlibreb}, Euz\'{e}bio et al. identified parameter conditions in \eqref{FZ} leading to a Hopf-zero equilibrium at singular points. They subsequently employed the averaging theory to establish conditions for guaranteeing the existence of periodic solutions. Similarly, in \cite{periodicFN}, Vidal and Llibre employed the averaging theory to derive conditions ensuring periodic solutions in the periodic FitzHugh-Nagumo system.
 
In this paper, our focus is on system \eqref{FZ} around the origin when it exhibits a Hopf-zero equilibrium, characterized by the presence of a pair of purely complex conjugate eigenvalues and a single zero eigenvalue in the associated Jacobian matrix. 

The origin is an equilibrium of Hopf-zero type of system \eqref{FZ} for two families of parameters (see \cite{RodrigoLlibreb}), namely:

\begin{itemize}
\item
\textbf{Family A:} For $(a,b,c)=\left(-\dfrac{1}{d},\dfrac{1}{d}\sqrt{\dfrac{1}{d}-w^2},\sqrt{\dfrac{1}{d}-w^2}\right)$, $0<d\not\in \{0,1\} ,$ $w>0$, and $d(1-d w^2)>0,$ the Jacobian matrix of \eqref{FZ} at the origin has a zero eigenvalue and a pair of purely imaginary conjugate eigenvalues $\pm iw.$

\item \textbf{Family B:} For $(a,b,c)=\left(-w^2,0,0\right)$, $d\in\R$, and $w>0,$ the Jacobian matrix of \eqref{FZ} at the origin has a zero eigenvalue and a pair of purely imaginary conjugate $\pm iw.$
\end{itemize}

Our primary objective is to identify, besides the periodic solutions already obtained in \cite{RodrigoLlibreb}, invariant tori within the aforementioned families, as well as to investigate the stability properties of these invariant structures.
The torus bifurcation is obtained through the occurrence of a Neimark-Sacker bifurcation in the Poincaré map. Specifically, an invariant torus for a differential equation corresponds to an invariant closed curve in the Poincaré map, and a Neimark-Sacker bifurcation in a 1-parameter family of planar maps implies the creation of an invariant closed curve from a fixed point as its stability changes.

\subsection{Results on Family A} 
In \cite{RodrigoLlibreb}, for Family A, it was proven that a periodic solution is born at the origin as $\varepsilon$ passes through $0$. For the sake of completeness, we re-prove the existence of such a periodic solution and provide its stability. We also show that this solution converges to $(0,0,0)$ as $\varepsilon$ converges to zero. Furthermore, we establish the existence of a torus bifurcated from the periodic solution $\phi(t,\varepsilon)$. 

Assume that
\begin{small}
\begin{eqnarray}\label{abcA}
(a,b,c)=\left(-\frac{1}{d}+\varepsilon \alpha_1+\varepsilon^2 \alpha_2,\frac{1}{d}\sqrt{\frac{1}{d}-w^2}+\varepsilon \beta_1+\varepsilon^2 \beta_2,
\sqrt{\frac{1}{d}-w^2}+\varepsilon \gamma_1+\varepsilon^2 \gamma_2\right)+\mathcal{O}(\varepsilon^2),
\end{eqnarray}
\end{small}
where $0< d\not\in \{0,1\},$ $w>0$ and $\varepsilon, \alpha_1, \alpha_2, \beta_1,\beta_2, \gamma_1, \gamma_2 \in \mathbb{R},$ and define
\begin{align}\label{lazem}
l_0&:=2 w^4 d \left(d\beta_1-\gamma_1\right)^2+2{\alpha_1}^2 \left(d w^2-1\right),\\\nonumber
l_1&:=w d^2\beta_0 (15 w^2 d^2+41 d^2-82 d+41)-6\pi (d-1)^2 (w^2 d-1).
\end{align}

\begin{mtheorem}\label{teoA}
Let $(a,b,c)$ be given by \eqref{abcA}. Then, the following are valid:
\begin{itemize}
\item[I.]
If $l_0(w^2 d-1)>0$ and $l_0 \neq 0$, for $|\varepsilon| \neq 0$ small enough, the FitzHugh-Nagumo system \eqref{FZ} possesses a periodic solution $\phi(t,\varepsilon)$ such that $\phi(t,\varepsilon)\rightarrow (0,0,0)$ as $\varepsilon \rightarrow 0$. In addition, if $ \beta_1 d-\gamma_1>0$ (resp. $ \beta_1 d-\gamma_1<0$), then the  periodic solution $\phi$ is attracting (resp. repelling).

\item[II.] If $l_1 \neq 0$, then a smooth curve $\gamma(\varepsilon)=\beta_1 d+\mathcal{O}(\varepsilon)$ and an interval $J_{\varepsilon}$ containing $\gamma(\varepsilon)$ exist such that when $\gamma_1$ passes through $\gamma(\varepsilon)$, the FitzHugh-Nagumo system \eqref{FZ}  undergoes a bifurcation of a unique invariant torus from the periodic solution $\phi(t,\gamma(\varepsilon),\varepsilon)$, where $\phi(t,\gamma_1,\varepsilon)=\phi(t,\varepsilon)$. In addition, if $\gamma_1 \in J_{\varepsilon}$ and $ \gamma_1-\gamma(\varepsilon) >0$, such a torus exists, surrounds the periodic solution, and has opposite stability to the periodic solution.
\end{itemize}
\end{mtheorem}

\subsection{Results on Family B} 
Assume that
\begin{eqnarray}\label{abcB}
(a,b,c)=\left(-w^2+\alpha(\varepsilon),\beta(\varepsilon),\gamma(\varepsilon)\right)+\mathcal{O}(\varepsilon^4), \,\, 
\mbox{and}\,\, d\in\mathbb{R}
\end{eqnarray}
where 
\begin{eqnarray}\label{alphabetagamma}
\alpha(\varepsilon)=\sum_{i=1}^5 \varepsilon^i \alpha_i,\,\, \beta(\varepsilon)=\sum_{i=1}^5 \varepsilon^i \beta_i,\,\, \gamma(\varepsilon)=\sum_{i=1}^5 \varepsilon^i \gamma_i, \,\,
w>0,\,\,
and \,\, \alpha_i, \beta_i, \gamma_i, \varepsilon \in \mathbb{R}.
\end{eqnarray}
Also, define 
\begin{align}\label{valuesB}
k_0:=&(\gamma_1\alpha_1-w^2\gamma_2+\beta_2) (\gamma_1\alpha_1+w^2\gamma_2-\beta_2),\\\nonumber
\eta:=&\dfrac{(1-dw^2) (w^2{{ \gamma_1}}^{3}d -w^4{\gamma_1}^3  d^2
-w^2 \gamma_3+\alpha_1 \gamma_2+\alpha_2\gamma_1+\beta_3)}
{\gamma_1 ( 1-5w^2+w^4+3dw^2-3d w^4+3 d w^6 )},\\\nonumber
\lambda_1:=&\dfrac{2\gamma_1 \pi(1-dw^2)}{w},\\\nonumber
\lambda_2:=&\dfrac{2\pi \left(w^2\gamma_1^3 d-w^4\gamma_1^3 d^2-w^2\gamma_3+\gamma_1\alpha2+\alpha_1\gamma_2+\beta_3\right)}{w^3}.
\end{align}
 
We divide this family into two subclasses namely \textbf{Class B1} and \textbf{Class B2}, see Theorem \ref{teoB}. 

Under the conditions specified in \textbf{Class B1}, the existence of a periodic solution was previously established in \cite{RodrigoLlibreb}. This was achieved by showing that the first-order averaged function vanishes identically, while the second-order function has a non-trivial zero, $(r_{\star},z_{\star})$, corresponding to a periodic solution. In this work, we re-prove the existence of this periodic solution for completeness and, in addition, we provide a detailed analysis of its stability. Furthermore, we also demonstrate the existence of a torus bifurcating from this periodic solution.

Now, under the conditions of \textbf{Class B2}, to the best of our knowledge, the bifurcation of periodic solutions has not been explored. This class is characterized by a continuum of zeros for the first-order averaging function, which remains non-vanishing. To investigate the bifurcation of periodic solutions in this context, we employ recent advancements in averaging theory based on the Lyapunov-Schmidt reduction.

\begin{mtheorem}\label{teoB}
Let $(a,b,c)$ be defined by \eqref{abcB} and assume that $\beta_1=\gamma_1 w^2.$ we deal with two classes:
\begin{itemize}
\item[\textbf{Class B1.}]
Assume that $d=\dfrac{1}{w^2}$ and $\beta_2\neq w^2 \gamma_2 -\alpha_1 \gamma_1.$ 
Therefore, the following statements hold:
\begin{itemize}
\item[Q1.]
For $k_0>0$ and $|\varepsilon| \neq 0$ small enough, the FitzHugh-Nagumo system \eqref{FZ} possesses a periodic solution $\phi(t,\varepsilon)$ such that $\phi(t,\varepsilon)$ converges to $(0,0,0)$ as 
$\varepsilon$ converges to zero. Furthermore, a periodic solution of this nature is attracting (resp. repelling) if $\beta_2<\gamma_2 w^2$ (resp. $\beta_2>\gamma_2 w^2$).
\item[Q2.]
If $(w^2-1) \gamma_1 \neq 0$, there is a smooth curve $\beta(\varepsilon)=\gamma_2 w^2+\mathcal{O}(\varepsilon)$ and interval $J_{\varepsilon}$  containing $\beta(\varepsilon)$ while a unique invariant torus
emerges
from the periodic solution $\phi(t,\beta(\varepsilon),\varepsilon)$ when $\beta_2$ passes over $\beta(\varepsilon).$
This torus exists while surrounds the periodic solution when $\beta_2 \in J_{\varepsilon}$ and $\beta_2-\beta(\varepsilon) <0.$ Note that $\phi(t,\beta_2,\varepsilon)=\phi(t,\varepsilon).$ 
Additionally, the torus is attracting (resp. repelling) when encircling the unstable (resp. stable) periodic orbit if $l_{1,3}<0$ (resp. $l_{1,3}>0$.)

\end{itemize}
\item[\textbf{Class B2.}]
Assume that $\beta_2= w^2 \gamma_2 -\alpha_1 \gamma_1$ and $d\neq \dfrac{1}{w^2}.$ So, the following statements hold:
\begin{itemize}
\item[Q1.]
For $|\varepsilon| \neq 0$ enough small, system \eqref{FZ} possesses a periodic solution $\phi(t,\varepsilon)$ such that $\phi(t,\varepsilon)\rightarrow (0,0,0)$ as $\varepsilon \rightarrow 0$.
\item[Q2.]
If $\lambda_{1,2}>0$ (resp. $\lambda_{1,2}<0$) then the generated periodic solution is unstable (resp. asymptotically stable) while it is unstable when $\lambda_1\lambda_2<0.$
\end{itemize}
\end{itemize}
\end{mtheorem}

\subsection{Structure of the Paper}
Section 2 is divided into four subsections. In the first part, subsection 2.1, we review the averaging theory and present the key formulae needed throughout the other sections. In the second part, subsection 2.2, we discuss the existence of periodic solutions via Lyapunov-Schmidt reduction, highlighting that this method is an efficient alternative for detecting periodic solutions when the zero of the averaging function is not isolated. Subsection 2.3 recalls theorems related to averaging theory and $k$-determined hyperbolic matrices in the context of stability. Finally, subsection 2.4 examines the general requirements for averaging functions to ensure the occurrence of a torus bifurcation. Section 3 is separated into three subsections, investigating the existence of a periodic solution, the stability analysis of that periodic solution, and the generic conditions for the existence and stability of a torus bifurcation for Family A, respectively. Section 4 also contains three subsections. The first two subsections discuss the existence and stability of periodic solutions, while the last part studies the existence of a torus bifurcation. In section 5, We employ the higher-order backward differentiation formula to solve stiff ordinary differential equations, providing numerical examples for each result discussed in Theorems A and B.
In conclusion, Section 6 provides an overview of the main findings.

\section{Preliminary Theorems and Formulae}
\subsection{Averaging Theory}
The averaging theory is among the most useful
and efficient methods for identifying periodic solutions.
In simple terms, such a method establishes a link between periodic solutions of a non-autonomous periodic differential equations and zeros of the averaged functions.

Consider the following standard form of a non-autonomous differential system as
\begin{eqnarray}\label{Standard}
\dot{\mathbf{x}}=\sum_{i=1}^k \mathbf{G}_i(t,\mathbf{x})  \varepsilon^i+ \tilde{\mathbf{G}}(t,\mathbf{x} , \varepsilon) \varepsilon^{k+1}, \,\, (t,\mathbf{x} , \varepsilon)\in \mathbb{R}\times \Omega \times (-\bar{\varepsilon}, \bar{\varepsilon})
\end{eqnarray}
in which $\Omega \subseteq\mathbb{R}^n$ is an open bounded set and $\bar{\varepsilon}>0$ small. Furthermore, functions $\mathbf{G}_i$ and $\tilde{\mathbf{G}}$ in the variable $t$ are $T$- periodic.   Since this system is periodic, we can describe it in the extended phase space $\mathbb{R}/T\mathbb{Z} \times \Omega.$

The averaging method introduces a sequence of  averaged functions of order $i$, $1\leq i\leq k$, which are obtained via
\[
g_i (\mathbf{z})=\dfrac{y_i (T,\mathbf{z})}{i!},
\]
where
\begin{eqnarray*}
y_1(t,\mathbf{z})&=&\int _0^t \mathbf{G}_1 (\tau,\mathbf{z}) d\tau, \\\nonumber
y_2(t,\mathbf{z})&=&2 \int _0^t \left(\mathbf{G}_2 (\tau,\mathbf{z})+\dfrac{\partial\mathbf{G}_1}{\partial\mathbf{x}}(\tau,\mathbf{z})y_1(\tau,\mathbf{z})\right) d\tau, \\\nonumber
y_3(t,\mathbf{z})&=&6 \int _0^t \left(\mathbf{G}_3 (\tau,\mathbf{z})+\dfrac{\partial\mathbf{G}_2}{\partial\mathbf{x}}(\tau,\mathbf{z})y_1(\tau,\mathbf{z})+\dfrac{1}{2}\dfrac{\partial^2\mathbf{G}_1}{\partial\mathbf{x}^2}(\tau,\mathbf{z})y_1(\tau,\mathbf{z})^2+\dfrac{1}{2}\dfrac{\partial\mathbf{G}_1}{\partial\mathbf{x}}(\tau,\mathbf{z})y_2(\tau,\mathbf{z})\right) d\tau,\\\nonumber
y_4(t,\mathbf{z})&=&\int_0^t \bigg(
24\mathbf{G}_4 (\tau,\mathbf{z})+24\dfrac{\partial\mathbf{G}_3}{\partial\mathbf{x}}(\tau,\mathbf{z})y_1(\tau,\mathbf{z})+12 \dfrac{\partial^2\mathbf{G}_2}{\partial\mathbf{x}^2}(\tau,\mathbf{z})y_1(\tau,\mathbf{z})^2
+12\dfrac{\partial\mathbf{G}_2}{\partial\mathbf{x}}(\tau,\mathbf{z})y_2(\tau,\mathbf{z})\\\nonumber
&&+12\dfrac{\partial^2\mathbf{G}_1}{\partial\mathbf{x}^2}(\tau,\mathbf{z})y_1(\tau,\mathbf{z})\odot y_2(\tau,\mathbf{z})+4\dfrac{\partial^3\mathbf{G}_1}{\partial\mathbf{x}^3}(\tau,\mathbf{z})y_1(\tau,\mathbf{z})^3
+4\dfrac{\partial\mathbf{G}_1}{\partial\mathbf{x}}(\tau,\mathbf{z})y_3(\tau,\mathbf{z})^3
\bigg) d\tau,\\\nonumber
y_5(t,\mathbf{z})&=&\int_0^t \bigg(
120\mathbf{G}_5 (\tau,\mathbf{z})+120\dfrac{\partial\mathbf{G}_4}{\partial\mathbf{x}}(\tau,\mathbf{z})y_1(\tau,\mathbf{z})+60 \dfrac{\partial^2\mathbf{G}_2}{\partial\mathbf{x}^2}(\tau,\mathbf{z})y_1(\tau,\mathbf{z})^2
+60\dfrac{\partial\mathbf{G}_3}{\partial\mathbf{x}}(\tau,\mathbf{z})y_2(\tau,\mathbf{z})\\\nonumber
&&+60\dfrac{\partial^2\mathbf{G}_2}{\partial\mathbf{x}^2}(\tau,\mathbf{z})y_1(\tau,\mathbf{z})\odot y_2(\tau,\mathbf{z})+20\dfrac{\partial^3\mathbf{G}_2}{\partial\mathbf{x}^3}(\tau,\mathbf{z})y_1(\tau,\mathbf{z})^3
+20\dfrac{\partial\mathbf{G}_2}{\partial\mathbf{x}}(\tau,\mathbf{z})y_3(\tau,\mathbf{z})\\\nonumber
&&+20 \dfrac{\partial^2\mathbf{G}_1}{\partial\mathbf{x}^2}(\tau,\mathbf{z})y_1(\tau,\mathbf{z})\odot y_3(\tau,\mathbf{z})+15\dfrac{\partial^2\mathbf{G}_1}{\partial\mathbf{x}^2}(\tau,\mathbf{z})y_2(\tau,\mathbf{z})^2
+5\dfrac{\partial^4\mathbf{G}_1}{\partial\mathbf{x}^4}(\tau,\mathbf{z})y_1(\tau,\mathbf{z})^4\\\nonumber
&&+30\dfrac{\partial^3\mathbf{G}_1}{\partial\mathbf{x}^3}(\tau,\mathbf{z})y_1(\tau,\mathbf{z})^2\odot y_2(\tau,\mathbf{z})+5\dfrac{\partial\mathbf{G}_1}{\partial\mathbf{x}}(\tau,\mathbf{z})y_4(\tau,\mathbf{z})\bigg) d\tau.
\end{eqnarray*}
Note that in above notation, $\bigodot^l _{j=1} y_j$ denotes $l-$ Frechet derivative of
$\mathbf{G}_i(t,\mathbf{x})$ and it is expressed as
\[
\dfrac{\partial^l \mathbf{G}_i}{\partial \mathbf{x}^l} (t, \mathbf{x}) \bigodot^l _{j=1} y_j=
 \left( \sum^{n}_{\substack{i_r=1\\r=1..l}} \dfrac{\partial^l \mathbf{G}_i^1}
 {\partial \mathbf{x}_{i_1}\cdots \partial \mathbf{x}_{i_l}} (t, \mathbf{x}) y_1^{i_1}\cdots y_l^{i_l}, \cdots, \sum^{n}_{\substack{i_r=1\\r=1...l}} \dfrac{\partial^l \mathbf{G}_i^n}{\partial \mathbf{x}_{i_1}\cdots \partial \mathbf{x}_{i_l}} (t, \mathbf{x}) y_1^{i_1}\cdots y_l^{i_l} \right).
\]
\begin{lemma}\label{lemma}
\cite{NovaesBrouwer}: Let $x(0,\mathbf{x},\varepsilon): [0, t_z]\rightarrow \mathbb{R}^n$ is the solution for \eqref{Standard}. Then, for $|\varepsilon|$ sufficiently small and $T<t_{z},$ we have
\[
x(t,\mathbf{x},\varepsilon) =\mathbf{x}+\sum_{i=1}^k 
\dfrac{y_i (t, \mathbf{x})}{i!} \varepsilon^i+\varepsilon^{k+1}\mathcal{O}(1).
\]
\end{lemma}
Here $l\in [1,k],$ represents the subindex of the first averaged function which does not vanish. From Lemma \ref{lemma}, 
the Poincar\'e map on the transversal section $\sum=\{0\}\times \Omega$ for the differential equation \eqref{Standard} is given by
\begin{eqnarray}\label{poinc0}
\mathbf{x}\mapsto P(\mathbf{x},\varepsilon)=\mathbf{x}+g_l(\mathbf{x})+\varepsilon^l g_{l+1}(\mathbf{x})+\cdots+\varepsilon^{k-l} g_{k}(\mathbf{x})+\mathcal{O}(\varepsilon^{k+1-l}).
\end{eqnarray}
Two key observations are: periodic solutions correspond to fixed points of the Poincar\'e map \eqref{poinc0}; while invariant tori correspond to invariant closed curves of a Poincar\'e map \eqref{poinc0}.

The following classical result follows directly from \eqref{poinc0}. 
\begin{theorem}\label{thm1}
\cite{NovaesBrouwer}. Take into account the standard form \eqref{Standard} and assume $1\leq m \leq k$ be the first index such that $g_m \neq 0$ and $\mathbf{z}^{\ast}\in \Omega$ be the solution of $g_m$, i.e.  $g_m(\mathbf{z}^{\ast})=0.$ Further for $|\varepsilon| \neq 0$ enough small, let $\left|Dg_{m}(\mathbf{z}^{\ast})\neq 0 \right|.$ Then, there is 
a single T-periodic solution $\phi(t,\varepsilon)$ where  $\phi(0,0)=\mathbf{z}^{\ast}.$
\end{theorem}
Theorem \ref{thm1} provides that the simple zeros of averaged functions are connected to isolated periodic solutions of the differential system.

\subsection{Periodic solutions  in view of Lyapunov-Schmidt reduction}\label{secLya}

In \cite{Lyapunov}, taking advantage of Lyapunov-Schmidt reduction techinique, the authors outline conditions sufficient for the existence of periodic solutions arising from non-isolated zero sets (see also, \cite{buica07,LlibreNovaes}). In what follows, we briefly point out the main results, without mentioning the details.

Suppose that function $g_1$
is zero on
\[
\mathcal{Z}=\{ \mathbf{z}_u=(u, \mathcal{B}(u)): u \in \overline{V}\} \subset U,
\]
where open bounded set $V\subseteq\mathbb{R}^m$ and $\mathcal{C}^k$-function $\mathcal{B}:\overline{V}\rightarrow \mathbb{R}^{n-m}.$ Denote 
\[
Dg_1(\mathbf{z}_u)=\begin{pmatrix}
\Lambda_u & \Gamma_u \\ 
B_u & \Delta_u
\end{pmatrix},
\]
and let $f_i$ refer to bifurcation function of order $i$ are defined as
\begin{eqnarray*}
\gamma_1(u)&=&-\Delta_u^{-1} \pi^{\perp} g_2(\mathbf{z}_u),\\\nonumber
f_1(u)&=&\Gamma_u \gamma_1(u)+\pi g_2(\mathbf{z}_u),\\\nonumber
\gamma_2(u)&=&-\Delta_u^{-1} \left(\frac{\partial^2 \pi^{\perp} g_1}{\partial b^2}(\mathbf{z}_u) \gamma_1(u)^2+2
\frac{\partial\pi^{\perp} g_2}{\partial b}(\mathbf{z}_u) \gamma_1(u)+2\pi^{\perp}g_3(u)\right),\\\nonumber
f_2(u)&=&\frac{1}{2}\Gamma_u \gamma_2(u)+\frac{\partial^2\pi g_1}{2\partial b^2}(\mathbf{z}_u)\gamma_1(u)^2
+\frac{\partial \pi g_2}{\partial b}(\mathbf{z}_u)\gamma_1(u)+\pi g_3(\mathbf{z}_u),\\\nonumber
\gamma_3(u)&=&-\Delta_u^{-1} \bigg(
\frac{\partial^3 \pi^{\perp} g_1}{\partial b^3}(\mathbf{z}_u)\gamma_1(u)^3
+3\frac{\partial^2 \pi^{\perp} g_1}{\partial b^2}(\mathbf{z}_u)\gamma_1(u)\odot \gamma_2(u)
+3\frac{\partial^2 \pi^{\perp} g_2}{\partial b^2}(\mathbf{z}_u)\gamma_1(u)^2\\\nonumber
&&+2\frac{\partial \pi^{\perp} g_2}{\partial b}(\mathbf{z}_u)\gamma_2(u)
+6\frac{\partial \pi^{\perp} g_3}{\partial b}(\mathbf{z}_u)\gamma_1(u)+6\pi^{\perp} g_4(u)\bigg),\\\nonumber
f_3(u)&=&\frac{1}{6}\Gamma_u \gamma_3(u)+\frac{\partial^3\pi g_1}{6\partial b^3}(\mathbf{z}_u)\gamma_1(u)^3
+\frac{\partial^2 \pi g_1}{2\partial b^2}(\mathbf{z}_u)\gamma_1(u)\odot\gamma_2(u)+\frac{\partial^2\pi g_2}{2\partial b^2}(\mathbf{z}_u)\gamma_1(u)^2\\\nonumber
&&+\frac{\partial \pi g_2}{2\partial b}(\mathbf{z}_u)\gamma_2(u)+\frac{\partial \pi g_3}{\partial b}(\mathbf{z}_u)\gamma_1(u)
+\pi g_4(\mathbf{z}_u),
\end{eqnarray*}
\begin{eqnarray*}
\gamma_4(u)&=&-\Delta_u^{-1} \bigg(
\frac{\partial^4 \pi^{\perp} g_1}{\partial b^4}(\mathbf{z}_u)\gamma_1(u)^4
+3\frac{\partial^2 \pi^{\perp} g_1}{\partial b^2}(\mathbf{z}_u)\gamma_2(u)^2
+4\frac{\partial^2 \pi^{\perp} g_1}{\partial b^2}(\mathbf{z}_u)\gamma_1(u)\odot\gamma_3(u)\\\nonumber
&&+6\frac{\partial^3 \pi^{\perp} g_1}{\partial b^3}(\mathbf{z}_u)\gamma_1(u)^2\odot\gamma_2(u)
+4\frac{\partial \pi^{\perp} g_2}{\partial b}(\mathbf{z}_u)\gamma_3(u)
+4\frac{\partial^3 \pi^{\perp} g_1}{\partial b^3}(\mathbf{z}_u)\gamma_1(u)^3\\\nonumber
&&+12\frac{\partial \pi^{\perp} g_3}{\partial b}(\mathbf{z}_u)\gamma_2(u)
+12\frac{\partial^2 \pi^{\perp} g_2}{\partial b^2}(\mathbf{z}_u)\gamma_1(u)\odot\gamma_2(u)
+12\frac{\partial^2 \pi^{\perp} g_3}{\partial b^2}(\mathbf{z}_u)\gamma_1(u)^2\\\nonumber
&&+24\frac{\partial\pi^{\perp}g_4}{\partial b}(\mathbf{z}_u)\gamma_1(u)\bigg),\\\nonumber
f_4(u)&=&\frac{1}{24}\Gamma_u \gamma_4(u)+\frac{\partial^4\pi g_1}{24\partial b^4}(\mathbf{z}_u)\gamma_1(u)^4
+\frac{\partial^3 \pi g_1}{4\partial b^3}(\mathbf{z}_u)\gamma_1(u)^2\odot\gamma_2(u)
+\frac{\partial^2\pi g_1}{8\partial b^2}(\mathbf{z}_u)\gamma_2(u)^2\\\nonumber
&&+\frac{\partial^2 \pi g_1}{6\partial b^2}(\mathbf{z}_u)\gamma_1(u)\odot\gamma_3(u)
+\frac{\partial^3 \pi g_2}{6\partial b^3}(\mathbf{z}_u)\gamma_1(u)^3
+\frac{\partial^2 \pi g_2}{2\partial b^2}(\mathbf{z}_u)\gamma_1(u)\odot\gamma_2(u)\\\nonumber
&&+\frac{\partial \pi g_2}{6\partial b}(\mathbf{z}_u)\gamma_3(u)
+\frac{\partial^2 \pi g_3}{2\partial b^2}(\mathbf{z}_u)\gamma_1(u)^2
+\frac{\partial \pi g_3}{2\partial b}(\mathbf{z}_u)\gamma_2(u)
+\frac{\partial \pi g_4}{\partial b}(\mathbf{z}_u)\gamma_1(u)
+\pi g_5(\mathbf{z}_u).
\end{eqnarray*}

\begin{theorem}
\cite{Lyapunov}. Assume that for all $u \in \overline{V}$ we have $det(\Delta_u)\neq 0$ and $1\leq m \leq k-1$ is the first subindex that $f_m \neq 0$. 
When $u^{\ast}\in V$ and $\left(f_m(u^{\ast})=0, |Df_m(u^{\ast})|\neq 0\right)$ hold, 
system \eqref{Standard} possesses a single $T$-periodic solution for $|\varepsilon|\neq 0$ enough small denoted by $\phi(t,\varepsilon),$ where  $\phi(0,0)=\mathbf{z}_{u^{\ast}}$.
\end{theorem}

\subsection{Stability of Periodic Solution}
Theorem \ref{vehulst} and Theorem \ref{Murdock} present conditions that are sufficient to ensure the stability of periodic solutions
when the matrix $Dg_{m}(\mathbf{z}^{\ast})$ is hyperbolic and not hyperbolic, respectively.

\begin{theorem}\label{vehulst}
\cite{verhust}. Suppose the differential system \eqref{Standard} satisfied the hypothesis of the Theorem \ref{thm1}.  If all eigenvalues of $Dg_{m}(\mathbf{z}^{\ast})$ have negative real parts, for $|\varepsilon|>0$ enough small the periodic solution is attracting while it is repelling if there is an eigenvalue with a positive real part.
\end{theorem}

We recall that $K$-determined hyperbolicity refers to the determination of the hyperbolicity of an equilibrium point through the truncated linear part up to order $K$. 
Roughly speaking, a smooth matrix is $k-$hyperbolic whenever its hyperbolicity is $k$-determined.

\begin{theorem}\label{Murdock}
\cite{Murdock}. 
Let $D(\varepsilon)$ and $C(\varepsilon)$ be continuous matrix-valued functions 
given for
$0<\varepsilon$ and
\[
C(\varepsilon)=\Gamma(\varepsilon)+\varepsilon^S D(\varepsilon),
\]
where
\begin{equation*}
\Gamma(\varepsilon)=
\begin{pmatrix}
\eta_1(\varepsilon) & &  &  \\ 
 & \ddots &  &  \\ 
 &  & \ddots &  \\ 
 &  &  & \eta_n(\varepsilon)
\end{pmatrix} 
=\varepsilon^{s_1} \Gamma_1+ \cdots+\varepsilon^{s_j} \Gamma_j.
\end{equation*}
Then,  the eigenvalues of $C(\varepsilon)$ are the same as diagonal entries $\eta_i(\varepsilon)$ of $\Gamma(\varepsilon)$ 
with an error on the order of $\mathcal{O}(\varepsilon^S),$
when $0<\varepsilon< \varepsilon_0$ for $\varepsilon_0> 0.$
\end{theorem}

\subsection{Torus Bifurcation}
In this subsection, we point out the strategy introduced in \cite{NovaesJ}. 
This approach specifies the general conditions on the averaged functions required for the occurrence of Neimark-Sacker bifurcation in its associated Poincar\'e map. Specifically, the Neimark-Sacker bifurcation observed in the Poincaré map is commonly referred to as torus bifurcation.

Consider the 2-parameter family of non-autonomous differential system
\begin{eqnarray}\label{Standard2}
\dot{\mathbf{x}}=\sum_{i=1}^k \mathbf{G}_i(t,\mathbf{x};\mu) \varepsilon^i+ \tilde{\mathbf{G}}(t,\mathbf{x} ;\mu, \varepsilon) \varepsilon^{k+1}, \,\, (t,\mathbf{x};\mu , \varepsilon)\in \mathbb{R}\times \Omega \times J \times (-\bar{\varepsilon}, \bar{\varepsilon}),
\end{eqnarray}
and its Poincar\'e map as
\begin{eqnarray}\label{poincare}
\Pi(\mathbf{z},\mu,\varepsilon)=\mathbf{z}+\sum_{i=l}^{k} g_i(\mathbf{z},\mu) \varepsilon^i+\mathcal{O}(\varepsilon^{k+1}), \,\, \mathbf{x}(0,\mathbf{z},\mu,\varepsilon)=\mathbf{z},
\end{eqnarray}
where $l\in [1,k]$ is the subindex of averaged function which is non-vanished at first.

Consider $H_{\varepsilon}(y;\sigma)$ is obtained by doing the change of variables, parameters $\mathbf{x}=\mathbf{y}+\zeta(\mu,\varepsilon)$ and $\mu=\sigma+\mu(\varepsilon)$ to the Poincar\'e map
\eqref{poincare},i.e.
\begin{eqnarray*}
\mathbf{y}\rightarrow H_{\varepsilon}(\mathbf{y},\sigma)=\Pi(\mathbf{y}+\zeta(\sigma+\mu(\varepsilon),\varepsilon),\sigma+\mu(\varepsilon),\varepsilon),
\end{eqnarray*}
and multilinear functions
\begin{align}\label{BC}
\left(B^1_{\varepsilon}(\mathsf{u},\mathsf{v}),B^2_{\varepsilon}(\mathsf{u},\mathsf{v}) \right)&=
\left(\sum_{j, i=1}^{2} \dfrac{\partial^2 H^1_{\varepsilon}}{\partial x_i \partial x_j}(0,0) u_i v_j,
\sum_{j, i=1}^{2} \dfrac{\partial^2 H^2_{\varepsilon}}{\partial x_i \partial x_j}(0,0) u_i v_j \right),\\\nonumber
\left(C^1_{\varepsilon}(\mathsf{u},\mathsf{v}, \mathsf{w}),C^2_{\varepsilon}(\mathsf{u},\mathsf{v}, \mathsf{w}) \right)&=\left(\sum_{k,j,i=1}^{2} \dfrac{\partial^3 H^1_{\varepsilon}}{\partial x_i \partial x_j \partial x_k}(0,0) u_i v_j w_k,
\sum_{k,j,i=1}^{2} \dfrac{\partial^3 H^2_{\varepsilon}}{\partial x_i \partial x_j \partial x_k }(0,0) u_i v_j w_k \right).   
\end{align}

In addition, suppose the following hypotheses are valid.
\begin{itemize}
\item[\textbf{C1:}]  
A continuous curve $\mu \in J\mapsto \mathbf{x}_{\mu} \in \Omega$ exists within an interval $J,$ where $\mu\in J $ so that $g_l(\mathbf{x}_\mu; \mu)$ equals zero for all $\mu \in J$ and $D_x g_l(\mathbf{x}_m; \mu)$ possesses a pair of complex conjugate eigenvalues $\alpha(\mu)\pm i \beta(\mu),$ satisfying $\beta(\mu_0)=w_0>0$ and $\alpha(\mu_0)=0$ at a specific point $\mu_0.$ 
\begin{remark}
The hypothesis \textbf{C1} guaranties  the presence of a neighborhood $\bar{J}\subset J$ around $\mu_0$, a small parameter $\varepsilon_0\in (0,\bar{\varepsilon})$ and a unique function $\zeta:\bar{J}\times (-\varepsilon_0,\varepsilon_0)$ such that
\begin{eqnarray*}
\mbox{ for } (\mu,\varepsilon)\in \bar{J}\times(-\varepsilon_0,\varepsilon_0) \,\, \zeta(\mu,0)=\mathbf{x}_{\mu} \,\mbox{ and } \,\, \Pi (\zeta(\mu,\varepsilon),\mu,\varepsilon)=\zeta(\mu,\varepsilon).
\end{eqnarray*}
Further, it guaranties that the differential system \eqref{Standard2} has only a periodic orbit $\phi(t;\mu,\varepsilon)$ where $\phi(0;\mu,\varepsilon)$ converges $\mathbf{x}_{\mu}$ when 
$\varepsilon$ converges to zero.
\end{remark}

\item[\textbf{C2:}] The transversality condition 
\[
\dfrac{d\alpha(\mu)}{d\mu}|_{\mu=\mu_0}=d_0 \neq 0.
\]
\begin{remark}
Suppose that $\lambda(\mu,\varepsilon)$ and $\overline{\lambda(\mu,\varepsilon)}$ are the pair of 
eigenvalues for $D_{\mathbf{z}}\Pi(\zeta(\mu,\varepsilon),\mu,\varepsilon).$ Hypothesis $\textbf{C1}$ and $\textbf{C2}$ guaranty the existence of a small parameter $\varepsilon_1\in (0,\varepsilon_0)$ and a unique function 
$\mu:(-\varepsilon_1,\varepsilon_1)\rightarrow \bar{J}$ such that
\begin{equation*}
\mu(0)=\mu_0, \,\, |\lambda(\mu(\varepsilon),\varepsilon)|=1, \,\, \frac{d|\lambda(\mu,\varepsilon)|}{d\mu}\big|_{\mu=\mu(\varepsilon)}\neq 0, \,\, (\lambda(\mu(\varepsilon),\varepsilon))^k\neq 1\,\, k\in \{1,2,3,4\}.
\end{equation*}
\end{remark}
\item[\textbf{B3:}] The matrix $Id +\varepsilon^l \mathsf{A}_{\varepsilon}$ provided by equation $D_{\mathbf{y}} H_{\varepsilon}(0,0)=Id +\varepsilon^l \mathsf{A}_{\varepsilon}+\mathcal{O}(\varepsilon^{k+1})$ is in its real Jordan form, where $H_{\varepsilon}(y;\sigma)$ refers to the shifted Poincar\'e map and
\begin{align*}\label{B33}
D_{\mathbf{y}} H_{\varepsilon}(0,0)&=Id +\varepsilon^l \mathsf{A}_{\varepsilon}+\mathcal{O}(\varepsilon^{k+1})\\\nonumber
&=\begin{pmatrix}
1+\tilde{\alpha}(\varepsilon) & -\tilde{\beta}(\varepsilon) \\ 
\tilde{\beta}(\varepsilon) & 1+\tilde{\alpha}(\varepsilon)
\end{pmatrix} +\mathcal{O}(\varepsilon^{k+1}).
\end{align*}

\end{itemize}
\begin{theorem}\label{thmtorus}
\cite[Theorem \ref{teoB}]{NovaesJ} and \cite[Theorem 9]{NovaesN}. Consider the nonautonomous differential equation \eqref{Standard2} and assume that $j^{\ast}\in [l,k]$ is the first subindex so that $l_{1,j^{\ast}}\neq 0,$ where $l_1^{\varepsilon}$ points out to the Lyapunov coefficient
\begin{equation}\label{FormulaLyapunov}
\begin{aligned}
l_1^{\varepsilon}=&\mathfrak{Re}\left(\frac{1}{2} e^{-i \theta_{\varepsilon}}\left<\mathsf{p},C_{\varepsilon}(\mathsf{p},\mathsf{p},\bar{\mathsf{p}})\right>-\dfrac{e^{-2i \theta_{\varepsilon}}(1-2 e^{i \theta_{\varepsilon}}) 
}{2(1-e^{i \theta_{\varepsilon}})}\left<\mathsf{p},B_{\varepsilon}(\mathsf{p},\mathsf{p})\right>\left< \mathsf{p},
B_{\varepsilon}(\mathsf{p},\bar{\mathsf{p}})\right>\right)\\
&-\dfrac{\left| \left< \mathsf{p},B_{\varepsilon}(\bar{\mathsf{p}},\bar{\mathsf{p}})\right>\right|^2}{4}-\dfrac{\left| \left< \mathsf{p},B_{\varepsilon}(\mathsf{p},\bar{\mathsf{p}})\right>\right|^2}{2},
\end{aligned}
\end{equation}
where $\mathsf{p}=\left(\dfrac{1}{2},-\dfrac{i}{\sqrt{2}}\right),$ $e^{i \theta_{\varepsilon}}=1+\varepsilon (i w_0)+\mathcal{O}(\varepsilon^2)$ and the multilinear functions $B_{\varepsilon}(\mathsf{u},\mathsf{v})$ and $C_{\varepsilon}(\mathsf{u},\mathsf{v},\mathsf{w})$ are given by \eqref{BC}.
Then, for each $|\varepsilon|>0$ enough small, there are a map $\mu(\varepsilon) \in \mathbb{C}^1,$ $\mu(0)=\mu_0,$ and a neighborhood $U_{\varepsilon} \subseteq \Omega \times \mathbb{S}^1$ of the periodic solution $\phi(t;\mu(\varepsilon),\varepsilon)$ and  a neighborhood $J_{\varepsilon}$ of $\mu(\varepsilon)$ so that, we have:

\begin{itemize}
\item[i.] Let $\mu\in J_{\varepsilon} $ 
and
$l_{1,j^{\ast}} (\mu-\mu(\varepsilon)) d_0\geq 0.$ 
The stability of periodic orbit is determined by the sign of $l_{1,j^{\ast}},$ i.e. it is repelling (resp. attracting) if $l_{1,j^{\ast}} >0$ (resp. $l_{1,j^{\ast}}<0$). Further, 
within the neighborhood $U_{\varepsilon}$, system \eqref{Standard2} does not exhibit any invariant tori.

\item[ii.] 
System \eqref{Standard2} features a unique invariant torus, $T_{\mu,\varepsilon}\in U_{\varepsilon},$ encircling the periodic solution $\phi(t;\mu(\varepsilon),\varepsilon)$
when $\mu\in J_{\varepsilon} $ and  $l_{1,j^{\ast}} (\mu-\mu(\varepsilon)) d_0< 0.$
Further, such a torus is repelling (resp. attracting) for $l_{1,j^{\ast}}>0$ (resp. $l_{1,j^{\ast}}<0$) and the periodic solution has opposite stability to torus.

\item[iii.]  When $\mu$ crosses through the curve $\mu(\varepsilon),$ the differential equation \eqref{Standard2} has only one invariant torus, $T_{\mu,\varepsilon},$ in the neighborhood  $U_{\varepsilon}$ which is bifurcated from the periodic orbit $\phi(t;\mu(\varepsilon),\varepsilon).$

\end{itemize}
\end{theorem}
\begin{remark}
Here, 
$\langle\mathsf{u},\mathsf{v}\rangle
=\bar{\mathsf{u}}^T·\mathsf{v},$ for $\mathsf{u}, \mathsf{v} \in \mathbb{C}^2.$
\end{remark}
\section{Proof of Theorem \ref{teoA}}
\subsection{Existence of periodic solutions}

For $(a,b,c)=\left(-\dfrac{1}{d},\beta_0,d \beta_0\right)$ where $\beta_0=\dfrac{1}{d}\sqrt{\dfrac{1}{d}-w^2},$ the FitzHugh-Nagumo has a Hopf-zero singularity at the origin, then we can write the system \eqref{FZ} in the standard form \eqref{Standard}.
Let $(a,b,c)=\left(-\frac{1}{d}+\alpha_1 \varepsilon+\alpha_2 \varepsilon^2, \beta_0+\beta_1 \varepsilon+\beta_2 \varepsilon^2, d \beta_0+\gamma_1 \varepsilon+\gamma_2 \varepsilon^2\right)+\mathcal{O}(\varepsilon^3).$ Then, we have
\begin{align*}
\dot{x}&=z+\mathcal{O}(\varepsilon^3),\\\nonumber
\dot{y}&=\beta_2 \left(x-d y\right) \varepsilon^2+\beta_1\left(x-d y\right)\varepsilon+\beta_0\left(x-d y\right)+\mathcal{O}(\varepsilon^3) 
,\\\nonumber
\dot{z}&=-\left(\alpha_2 x^2-\alpha_2x-\gamma_2 z\right) \varepsilon^2-
\left(\alpha_1x^2-\alpha_1 x-\gamma_1 z\right)\varepsilon
+x^3+\frac{1-d}{d} x^2-\frac{1}{d}x+y+d \beta_0 z+\mathcal{O}(\varepsilon^3).
\end{align*}
By using the linear change of variables
\begin{eqnarray*}
(x,y,z)=\left(\dfrac{Y+2 Z(d w^2-1)}{2d w^2}, 
-\dfrac{d^2\beta_0 w X-(1-d w^2) (Y-2 Z)}{2 d^2 w^2}
,\dfrac{1}{2 w d} X\right),
\end{eqnarray*}
and taking $(X,Y,Z)=\varepsilon (x_1,x_2,x_3),$ we have
\begin{align*}
\dot{x}_1=&-wx_2+ \left(\gamma_1x_1+\frac { \left( x_2-2x_3+2w^2dx_3\right)  \left( 2\alpha_1d^2w^2+ \left( 1-d \right) \left( x_2-2x_3+2w^2dx_3 \right)  \right) }{2d^2w^3}\right) \varepsilon\\\nonumber
&+\scalebox{1.05}{$\left( \gamma_2x_1+{\frac { \left( x_2-2x_3+2{w}^2dx_3 \right) \alpha_2}w}+\frac { \left( x_2-2x_3+2w^2dx_3 \right) ^2 \left( x_2+2x_3 \left( w^2d-1 \right)-2\alpha_1w^2d \right) }{4 w^{5}d^2}\right) \varepsilon^2$}+\mathcal{O}(\varepsilon^3),\\\nonumber
\dot{x}_2=&w x_1-\Bigg(\dfrac{(x_2-2x_3+2w^2dx_3)(\alpha_1d^2w^2-w^2d^2x_3+w^2dx_3+dx_3-x_3-\frac{d}{2}x_2+\frac{1}{2}x_2)\beta_0}{d^2w^4}\\\nonumber
&+\dfrac{\beta_1d(wx_2+d\beta_0x_1)}{w}+\dfrac{\gamma_1x_1\beta_0}{w}\Bigg)\varepsilon
-\Bigg(\dfrac{\beta_2d(wx_2+d\beta_0x_1)}{w}
-\dfrac{(x_2-2x_3+2w^2dx_3)\alpha_2\beta_0}{w^2}\\\nonumber
&+\dfrac{\gamma_2x_1\beta_0}{w}+\dfrac{(x_2-2x_3+2w^2dx_3)^2 (2x_3-2w^2dx_3-x_2+2\alpha_1w^2d)\beta_0}{w^6d^2}\Bigg)\varepsilon^2+\mathcal{O}(\varepsilon^3),\\\nonumber
\dot{x}_3=&-\Bigg(\frac { \left( x_2-2x_3+2w^2dx_3\right)\left( 2\alpha_1 d^2w^2-2x_3+2dx_3-2w^2d^2x_3+2w^2dx_3
-dx_2+x_2\right)}{4\beta_0d^4w^4}\\\nonumber
&+\frac {\beta_1 \left( wx_2+d\beta_0x_1 \right) 
}{2wd^2{\beta_0}^2}+\frac{\gamma_1x_1}{2\beta_0wd^2} \Bigg) \varepsilon 
-\Bigg(\frac{\beta_2 \left( w y+d \beta_0 x_1 \right)}{2 w d^2{\beta_0}^2}
+\frac{\gamma_2x_1}{2wd^2\beta_0}+\frac {\left(x_2-2x_3+2w^2dx_3 \right) \alpha_2}{2w^2d^2\beta_0}\\\nonumber
&-\frac { \left( x_2-2x_3+2w^2dx_3 \right) ^2 \left( 
2x_3-2w^2dx_3-x_2+2\alpha_1w^2d \right) }{8w^6d^4\beta_0}
\Bigg)\varepsilon^2+\mathcal{O}(\varepsilon^3).
\end{align*}
Now, by applying cylindrical coordinate 
\begin{align*}
(x_1,x_2,x_3)=(r \cos{(\theta)},r \sin{(\theta)},z), \,\, \mbox{ where } \,\, \dot{\theta}=w+\mathcal{O}(\varepsilon)
\end{align*}
and doing a time rescaling, the FitzHugh–Nagumo system \eqref{FZ} changes into 
\begin{align*}
\frac{dr}{d\theta}&=\varepsilon F^1_1(\theta,r,z)+ \varepsilon^2 F^1_2(\theta,r,z)+\mathcal{O}(\varepsilon^3),\\\nonumber
\frac{dz}{d\theta}&=\varepsilon F^2_1(\theta,r,z)+ \varepsilon^2 F^2_2(\theta,r,z)+\mathcal{O}(\varepsilon^3).
\end{align*}
Since the expression of $F^i_j(\theta,r,z)$ for $i,j=1,2$ are too long, we skip them here.
The first-order averaged function is given by
\begin{align}\label{AVa}
g^1_1(r,z)&=\int_0^{2\pi} F^1_1(\theta,r,z) d\theta=
-\dfrac{r\pi\left(\left(\alpha_1d^2w^2-2z(d-1)(w^2d-1)\right)d\beta_0+d^2w^4(\beta_1d-\gamma_1)\right)}{d^2w^5},\\\nonumber
g^2_1(r,z)&=\int_0^{2\pi} F^2_1(\theta,r,z) d\theta=
\dfrac{\beta_0\pi \left((1-d) r^2-(8 (d-1)) (d w^2-1)^2 z^2+8\alpha_1 d^2 w^2 (dw^2-1) z\right)}{4d w^5 (dw^2-1)},
\end{align}

and the second-order averaged function is given by
\begin{align*}
g^1_2(r,z)=&\dfrac{\pi}{48d^{5}w^{10}}\bigg(dw\sqrt{\frac{1}{d}-w^2} \bigg((30d-15d^2-9d^2w^2-15)r^3
+16d^2w^3 (d-1) (4d\beta1-\gamma_1)r^2\\\nonumber
&+ (-144 (d^2w^2+3d^2-6d+3) 
(dw^2-1)^2 z^2-48 (dw^2-1) d^2w^2 ( 2\pi  (d-1) 
(d\beta_1-\gamma_1) w\\\nonumber
&-\alpha_1 (2dw^2+9d-9)) z
+24d^{4}w^{4} (({\gamma_1}^2-2{\alpha_2}-d^2{\beta_11}^2 ) w^2+2\pi {\alpha_1} (d\beta_1-\gamma_1) w-3{\alpha_1}^2)) r\\\nonumber
&\scalebox{0.98}{$-96d^2w^3 (d-1)
( dw^2-1 )^2 (d\beta1-\gamma_1) z^2+96\alpha_1d^{4}w^{5}
( dw^2-1) (d\beta1-\gamma_1) z\bigg)-576(d-1)^2$}\\\nonumber
& (dw^2-1)^{4}z^3
+864d^2w^2\alpha_1(d-1)(dw^2-1)^3 z^2-16 (dw^2-1)+\bigg(18d^{4}{\alpha_1}^2w^{4}
(dw^2-1)
\\\nonumber
&\scalebox{0.96}{$-5(d-1)^2(dw^2-1) r^2-3(d-1)d^2w^3 
(3d\beta_1-4d\gamma_1w^2+3\gamma_1) r\bigg) z+12\pi r^3 (d-1)^2(dw^2-1)$}\\\nonumber
&  -40d^2w^2\alpha_1 r^2(d-1) ( dw^2-1)
-24d^{4}w^{4}(2d w^{5} (\beta2 d-\gamma_2)-\pi d w^{4} (d\beta1-\gamma_1)^2 \\\nonumber
&
-4d\gamma_1\alpha_1 w^3+d{\alpha_1}^2\pi w^2+3\alpha_1 w(\gamma_1+d\beta1) -{\alpha_1}^2\pi) r\bigg),\\\nonumber
g^2_2(r,z)=&\frac{1}{24 d^5 w^{10}(d w^2-1) }\bigg(3d^2 w \sqrt{\frac{1}{d}-w^2}\big(16 (d^2w^2+3d^2-6d+3)  (dw^2-1)^3 z^3\\\nonumber
&-8d^2
w^{2}\alpha_1 ( 2dw^2+9d-9)(dw^2-1)^2z^2+2 (dw^2-1)  
((3d^2w^2+5d^2-10d+5) r^2\\\nonumber
&-4d^2w^3 (d-1) (4d\beta_1-\gamma_1) r+4d^{4}w^{4} ( 2w^2\alpha_2+3{\alpha_1}^2) ) z+ (2\pi d^2 (d-1)  (d\beta_1-\gamma_1) w^3
\\\nonumber
&-2d^3w^{4}\alpha_1-5d^2\alpha_1 (d-1) w^2) r^{2}+4\alpha_1 d^{4}w^{5} (4d\beta_1-\gamma_1) r\big) -96\pi (d-1) ^2(dw^2-1) ^{4}z^3\\\nonumber
&+24 (d-1)(dw^2-1)^2 ( d^3 ( d\beta_1+3\gamma_1) w^{5}+6d^3w^{4}\alpha_1\pi -3d^2 (d\beta_1+\gamma_1) w^3-3d ( 2d\alpha_1\pi \\\nonumber
&+dr-r) w^2+3r (d-1) ) z^2-24 (dw^2-1) d^2w^2\alpha_1( d^3 ( d\beta_1+3\gamma_1) {w}^{5}+2d^3w^{4}\alpha_1\pi\\\nonumber
& -3d^2 ( d\beta_1+\gamma_1) w^3-d ( 3dr+2d\alpha_1\pi -3r) w^2+3r (d-1) ) z-r (10 (d-1) ^2 (1-dw^2) r^2\\\nonumber
&-3(d-1) d^2w^3 ( dw^2 ( 3d\beta_1+5\gamma_1) -5d\beta_1-5\gamma_1) r+12d^{4}w^{4}\big( dw^2 ( dw^2{\beta_1}\gamma_1
-{\beta_1}^2w^2d^2 \\\nonumber
&+3{\alpha_1}^2\big)-3{\alpha_1}^2) ) \bigg).
\end{align*}
\begin{proposition}\label{propEcaseA}
Let 
\begin{eqnarray*}
(a,b,c)=\left(-\frac{1}{d}+\varepsilon \alpha_1+\varepsilon^2 \alpha_2,\frac{1}{d}\sqrt{\frac{1}{d}-w^2}+\varepsilon \beta_1+\varepsilon^2 \beta_2,\sqrt{\frac{1}{d}-w^2}+\varepsilon \gamma_1+\varepsilon^2 \gamma_2\right)+\mathcal{O}(\varepsilon^3),
\end{eqnarray*}
where $\varepsilon, \alpha_1, \beta_1, \gamma_1 \in \mathbb{R}$.
For $l_0<0$ and $|\varepsilon| \neq 0$ enough small, system \eqref{FZ} has a periodic solution $\phi(t,\varepsilon)$ such that 
$\phi(t,\varepsilon)$ approaches to $(0,0,0)$ when $\varepsilon$ converges to $0.$
\end{proposition}

\begin{proof}
The first averaged function \eqref{AVa} has two non-trivial solutions $(r^{\pm}_{\ast},z_{\ast})$
\begin{eqnarray*}
\left(r^{\pm}_{\ast}, z_{\ast}\right)
=\left(\pm\dfrac{ d^2 }{d-1} \sqrt{\dfrac{l_0}{w^2 d-1}}w^2
,\dfrac{ d w^2\left(d\beta_0\alpha_1+w^2 d\beta_1-w^2\gamma_1 \right)}{2 \beta_0 (d-1) (w^2d-1)} \right).
\end{eqnarray*}
For $d>1$ the solution $\left(r^{+}_{\ast}, z_{\ast}\right)$ only within the domain of $g_1$ and for $d<1$ the solution $\left(r^{-}_{\ast}, z_{\ast}\right)$ only contained in the domain of $g_1.$ Since,
\begin{eqnarray*}
\mbox{det}\left(D{g}_1\left(r^{\pm}_{\ast}, z_{\ast}\right)\right)=-\dfrac{\pi^2 l_0}{d w^6}\neq 0,
\end{eqnarray*}
and by using Theorem \ref{thm1} the Proposition is proved.

Assume that assumptions Proposition \ref{propEcaseA} hold for system \eqref{FZ}, then going back through the change of variables 
\begin{eqnarray*}
(x_1,x_2,x_3)=(r \cos{(\theta)},r \sin{(\theta)},z), \,\,
(X,Y,Z)=\varepsilon (x_1,x_2,x_3), 
\end{eqnarray*}
and
\begin{eqnarray*}
(x,y,z)=\left(\frac{Y+2 Z(d w^2-1)}{2d w^2}, 
-\frac{d^2\beta_0 w X+(d w^2-1) (Y-2 Z)}{2 d^2 w^2}
,\frac{1}{2 d w} X\right),
\end{eqnarray*}
This system has the limit cycle:
\begin{align*}
x(t,\varepsilon)&=\left(\frac{r\sin(w t)}{2dw^2}+\frac{z (dw^2-1)}{dw^2}\right) \varepsilon +\mathcal{O}(\varepsilon^2),\\\nonumber
y(t,\varepsilon)&=\left( -\frac{r \beta_0 \cos(w t)}{2w}-\frac{r(dw^2-1)\sin(w t)}{2d^2w^2}+\dfrac{z(dw^2-1)}{d^2w^2} \right) \varepsilon +\mathcal{O}(\varepsilon^2),\\\nonumber
z(t,\varepsilon)&=\frac{r\cos(w t)}{2dw} \varepsilon +\mathcal{O}(\varepsilon^2).
\end{align*}
\end{proof}
\subsection{Stability of Periodic Solution}
Here, the periodic solution provided by the Propositions \ref{propEcaseA} will be analysed in view of stability.
\begin{proposition}\label{stabilityA}
Suppose $l_0 <0.$ The periodic solution generated by Proposition \ref{propEcaseA} is attracting (resp. repelling) if $ \beta_1 d-\gamma_1>0$ (resp. $<0)$.
\end{proposition}

\begin{proof}
For $l_0 <0$ and $d>1,$ $(r^{+}_{\ast},z_{\ast})$  is contained in the domain $\mathbb{R}^{+}\times \mathbb{R}$ while for $l_0 <0$ and $d<1,$ $(r^{-}_{\ast},z_{\ast})$  is held in the domain.

The characteristic polynomial for Jacobian matrix of $g_1(\mathsf{z})$  at $(r^{\pm}_{\ast},z_{\ast})$ given by
\begin{eqnarray*}
p(\lambda)=\lambda^2+\dfrac{2\pi (\beta_1 d-\gamma_1)}{w}\lambda-\dfrac{ \pi^2 l_0}{d w^6}.
\end{eqnarray*}
Since $-\dfrac{ \pi^2 l_0}{d w^6}>0,$ then according to Routh-Hurwitz criteria, the following results are derived. When 
$( \beta_1 d-\gamma_1)>0$ then the roots of $p(\lambda)$ are on the left hand side of the complex plane. Besides for $(\gamma_1-\beta_1 d)>0,$ the polynomial $p(\lambda)$ has two roots with a positive real part. Thus, the periodic solution is repelling.
\end{proof}

\subsection{Invariant torus bifurcation}
Next, we utilize the method introduced in \cite{NovaesJ} and establish general conditions on the averaging functions that guarantee the emergence of an invariant torus from the periodic solution outlined in Proposition \ref{propEcaseA}.

Define
\begin{eqnarray}\label{a1221}
\varpi:=\frac {\sqrt2\pi \left(2{\alpha_1}\left(d-1\right) ^2 \left(2dw^2-3\right)
{\beta_1}+dw^2\left(d^2{{\alpha_1}}^2-2 d^2{\alpha_2}
+4{\alpha_2}d-2{\alpha_2} \right) {\beta_0} \right) }{2w^{5}\left(d-1\right)^2}. 
\end{eqnarray}
\begin{proposition}\label{Torus}
If the assumptions of Proposition \ref{propEcaseA} are met by the FitzHugh-Nagumo system \eqref{FZ} and $l_1 \neq 0$ as defined in \eqref{lazem}, then a smooth curve $\gamma(\varepsilon)=\gamma^{\ast}+\varepsilon \gamma1+\mathcal{O}(\varepsilon^2)$ exists where $\gamma^{\ast}=\beta_1 d,$ and interval $J_{\varepsilon}$ is contained in $\gamma(\varepsilon),$ so as the only one invariant torus bifurcates from the periodic solution 
$\phi(t,\gamma(\varepsilon),\varepsilon)$ when $\gamma_1$ crosses $\gamma(\varepsilon).$
If $\gamma_1 \in J_{\varepsilon}$ and $ \gamma_1-\gamma(\varepsilon) <0,$ such this torus exists while surrounds the periodic solution. Moreover, since $l_{12}>0,$ the torus is repelling while the periodic solution has the opposite stability.
\end{proposition}

\begin{proof}
Assume $\gamma^{\ast}=\beta_1 d$ and $p(\lambda)$ denotes the characteristic polynomial of the Jacobian matrix $\dfrac{\partial g_1}{\partial \mathsf{z}}(\mathsf{z}_{\gamma_1},\gamma_1).$ The roots of $p(\lambda)$ is given by $\lambda(\gamma_1)=\alpha(\gamma_1)\pm i \beta(\gamma_1),$ where
\begin{eqnarray}\label{H1H22}
\alpha(\gamma^{\ast})=0,\,\, 
\beta(\gamma^{\ast})=\dfrac{\sqrt{2}\pi |\alpha_1| \sqrt{d (1-w^2 d)}}{d w^3}=w_0, \,\, d_0=\dfrac{d\alpha(\gamma_1)}{d\gamma_1}\bigg|_{\gamma_1=\gamma^{\ast}}=\dfrac{\pi}{w}>0.
\end{eqnarray}
So, the equations \eqref{H1H22} verify hypotheses $\textbf{C1}$ and $\textbf{C2}$.

By using the implicit function theorem and hypotheses $\textbf{C1}$ and $\textbf{C2}$, we can conclude that there exist functions $\mathsf{z}_1$ and $\gamma1$ such that
\begin{eqnarray*}
\zeta(\gamma_1,\varepsilon)=\mathsf{z}_{\gamma_1} 
+\varepsilon \mathsf{z}_1+\mathcal{O}(\varepsilon^2) \,\, \mbox{and} \,\, \mu(\varepsilon)=\gamma^{\ast}
+\varepsilon \gamma1+\mathcal{O}(\varepsilon^2).
\end{eqnarray*}
We ignore these functions here due to cumbersome expressions.

In this step, we apply the change of variable 
$\mathsf{z}=\mathsf{x}+\zeta(\mu,\varepsilon),$ the change of parameter $\gamma_1=\sigma+\beta_1 d,$ and the linear change of variable 
\begin{eqnarray*}
{y}\mathsf=\mathsf{V} \mathsf{x}=\begin{pmatrix}
\frac{1}{2} & v_{12} \\ 
0 & v_{22}
\end{pmatrix} \mathsf{x},
\end{eqnarray*}
where
\begin{eqnarray*}
 v_{12}&=&\left(\frac{\beta_1 d}{2w}+\frac{\sqrt {2}\beta_0
 \pi d\alpha_1}{4w^{3}}
-\frac{\sqrt {2} \alpha_1 \left(5d^{2}+3d^{2}w^{2}
-10d+5\right) }{16w^{2} \left(d-1\right) ^{2}}+\frac{\sqrt {2}\alpha_1\pi  \left( w^{2}d-1 \right)}{4\beta_0 w^{3}d^{2}}\right)\varepsilon,\\\nonumber
 v_{22}&=&\frac {\sqrt {2} \left( 3\sqrt {2}dw\beta_1+10\alpha_1 \right)  }{24\beta_0{w}^{3}{d}^{2}}\varepsilon
+ \dfrac{1}{4 (1-d{w}^{2} )},
\end{eqnarray*}
to the Poincar\'e map
\begin{eqnarray*}
\Pi(\mathsf{z},\mu,\varepsilon)=\mathsf{z}+\varepsilon g_1(\mathsf{z},\mu)+\varepsilon^2 g_2(\mathsf{z},\mu)+\mathcal{O}(\varepsilon^{3}),
\end{eqnarray*}
to obtain
\begin{eqnarray*}
 H_{\varepsilon}(\mathbf{y},\sigma)=\mathsf{V}^{-1} \Pi\left( \mathsf{V}\mathbf{y}+\zeta(\gamma_1+\mu(\varepsilon),\varepsilon),\gamma_1+\mu(\varepsilon),\varepsilon\right).
\end{eqnarray*}
Now, the Taylor expansion of 
\begin{eqnarray*}
 D_{\mathbf{y}} H_{\varepsilon}(0,0)=Id+\varepsilon\mathsf{A}_{1}+\varepsilon^2 \mathsf{A}_{2},
\end{eqnarray*}
in real Jordan form is as
\begin{align*}
\begin{pmatrix}
1 & 0 \\ 
0 & 1
\end{pmatrix} +
\varepsilon
\begin{pmatrix}
0 & -\dfrac{\sqrt{2}\pi\beta_0\alpha_1 d}{w^3} \\ 
\dfrac{\sqrt{2}\pi\beta_0\alpha_1 d}{w^3} & 0
\end{pmatrix} +\varepsilon^2
\begin{pmatrix}
-\dfrac{\left(\pi{\alpha_1}{\beta_0} d\right)^2}{w^6} & \varpi\\ 
-\varpi & -\dfrac{\left(\pi{\alpha_1}{\beta_0} d\right)^2}{w^6}
\end{pmatrix},
\end{align*}
where $\varpi$ is given by \eqref{a1221}.
Therefore, we obtain the map
\begin{equation*}
\mathsf{y}\mapsto H_{\varepsilon}(\mathsf{y},\sigma)=\left( H^1_{\varepsilon}(\mathsf{y},\sigma),H^2_{\varepsilon}(\mathsf{y},\sigma) \right).
\end{equation*}
In the last step, we should calculate the Lyapunov coefficient \eqref{FormulaLyapunov}.
So, we need to compute the multilinear functions by using formulas \eqref{BC} 
as
\begin{align*}
B_{\varepsilon}(u,v)
=&\scalebox{0.97}{$\left( \dfrac{\pi \beta_0 (1-d) (u_2 v_1+u_1 v_2)}{2 w^5 d }\varepsilon,
\dfrac{\pi\beta_0(d-1) ( 2u_2 v_2+u_1 v_1)}{2 w^5 d}\varepsilon
\right)+\mathcal{O}(\varepsilon^2),$}\\\nonumber
C_{\varepsilon}(u,v,w)
=&\Bigg(\dfrac{\epsilon^2\pi}{96 d^5 w^{10}}\bigg( 4w^2\left(  \left( 9v_1w_1\pi -10v_1w_2-10v_2w_1\right) u_1-2u_2 
\left(5v_1w_1-27v_2w_2 \right)\right) d^{3}\\\nonumber
&-4 \left(2w^2+1\right) 
\left(\left( 9v_1w_1\pi -10v_1w_2-10v_2w_1 \right) u_1-2u_2
\left(5v_1w_1-27v_2w_2 \right)\right) d^2\\\nonumber
&+4 \left(w^2+2\right) 
\left(\left( 9v_1w_1\pi -10v_1w_2-10v_2w_1\right) u_1-2u_2
\left(5v_1w_1-27v_2w_2 \right)\right) d\\\nonumber
&+\left(40v_2w_1-36v_1w_1\pi +40v_1w_2\right) u_1+8u_2 
\left(5v_1w_1-27v_2w_2 \right)\\\nonumber
&-9d^2w 
\big(\left(w_1 \left(5+3w^2d^2-10d+5d^2\right) v_1
+4v_2w_2 \left(w^2d^2+3d^2-6d+3\right)\right) u_1\\\nonumber
&+4u_2w_2 \left(w^2d^2+3d^2-6d+3 \right) v_1
+4u_2v_2w_1 \left(w^2d^2+3d^2-6d+3\right)\big) {\beta_0}\bigg),\\\nonumber
&-\dfrac{\pi \epsilon^2}{8 d^5 w^{10}}\bigg(2w^2 \left(\left( 5v_1w_1-3v_2w_2 \right) 
u_1+3u_2 \left( 2v_2w_2\pi -v_2w_1-v_1w_2 \right)\right) d^{3}\\\nonumber
&-2 \left( 1+2w^2 \right) \left(  \left( 5v_1w_1-3v_2w_2 \right) u_1
+3u_2 \left( 2v_2w_2\pi -v_2w_1-v_1w_2\right)  \right) d^2\\\nonumber
&+2 \left( w^2+2 \right)  \left( \left( 5v_1w_1-3v_2w_2 \right) u_1
+3u_2\left( 2v_2w_2\pi -v_2w_1-v_1w_2 \right) \right) d\\\nonumber
&+6u_2v_1w_2
+ \left( 6v_2w_2-10v_1w_1 \right) u_1-6u_2v_2 \left( 2w_2\pi 
-w_{{1}} \right)\\\nonumber
&-d^2w \bigg(\big(\left( 3w^2v_1w_2
+5v_2w_1+3w^2v_2w_1+5v_1w_2 \right) u_1+ \left( 5u_2w_1+3w^2u_2w_1 \right) v_1\\\nonumber
&+6w^2u_2v_2w_2+18u_2v_2w_2\big) 
d^2+\left(-\left( 10v_1w_2+10v_2w_{{1}} \right) u_1-2u_2\left( 18v_2w_2+5v_1w_{{1}} \right)\right) d\\\nonumber
&+\left( 5v_1w_2+5v_2w_1\right)
u_1+u_2 \left( 18v_2w_2+5v_1w_1\right)  \bigg) { \beta_0}
\bigg)\Bigg)+\mathcal{O}(\varepsilon^3).
\end{align*}
Now, by expanding $l^1_{\varepsilon}$ around $\varepsilon=0,$ we have
\begin{equation*}
l_1^{\varepsilon}=
\varepsilon l_{1,1}+\varepsilon^2 l_{1,2}+\mathcal{O}(\varepsilon^3),
\end{equation*}
where 
\[
l_{1,1}=0, \, \mbox{ and } \, l_{1,2}=
\dfrac{\beta_0 (15 w^2 d^2+41 d^2-82 d+41)\pi}{128 d^3 w^9}-\dfrac{3\pi^2 (d-1)^2 (w^2 d-1)}{64 d^5 w^{10}}=\dfrac{\pi}{128 d^5 w^{10}}l_1,\,\,
\]
where $l_1$ is given by \eqref{lazem}. Then,
The existence and stability of this torus are determined by using the Theorem \ref{thmtorus}.
\end{proof}

\section{Proof of Theorem \ref{teoB}}
Since for  $(a,b,c)=\left(-w^2,0,0\right)$ where $w>0,$ the origin exhibits a Hopf-zero singularity in the FitzHugh-Nagumo system \eqref{FZ}. Hence, we are allowed to write it in the standard form
\eqref{Standard}.

Let $(a,b,c)=\left(-w^2+\sum_{i=1}^5 \varepsilon^i \alpha_i ,\sum_{i=1}^5 \varepsilon^i \beta_i,\sum_{i=1}^5 \varepsilon^i \gamma_i \right)+\mathcal{O}(\varepsilon^6).$ Then, we have
\begin{eqnarray*}
\dot{x}&=&z,\\\nonumber
\dot{y}&=&\beta_5 (x-d y)\varepsilon^5+\beta_4 (x-d y)\varepsilon^4+\beta_3(x-dy)\varepsilon^3+\beta_2(x-dy)\varepsilon^2+\beta_1(x-dy)\varepsilon+\mathcal{O}(\varepsilon^6),\\\nonumber
\dot{z}&=&\left(x(1-x)\alpha_5+\gamma_5 z\right)\varepsilon^5
+(x(1-x)\alpha_4+\gamma_4 z)\varepsilon^4+(\alpha_3 x-\alpha_3 x^2+\gamma_3 z)\varepsilon^3\\\nonumber
&&+(\alpha_2 x+\gamma_2 z-\alpha_2 x^2)\varepsilon^2+(\alpha_1 x+\gamma_1 z-\alpha_1x^2)\varepsilon+x^3+(w^2-1) x^2-w^2 x+y+\mathcal{O}(\varepsilon^6).
\end{eqnarray*}

By using the linear change of variables
\begin{equation*}
(x,y,z)=\left(\dfrac{2Z+w^2 Y}{2w^2}, Z,\dfrac{w}{2} X\right), \,\, \mbox{ and } \,\, (X,Y,Z)=\varepsilon (x_1,x_2,x_3),
\end{equation*}
we get 
\begin{align*}
\dot{x_1}=&-wx_2+ \left({\gamma_1}x_1+\frac { \left(2x_3+w^2x_2\right) \left(2\alpha_1 w^2+2w^2x_3+w^4x_2-2x_3-w^2 x_2\right)}{2w^5} \right) \varepsilon\\\nonumber
&+\left({\gamma_2}x_1-\frac { \left( 2x_3+x_2{w}^{2} \right)^2\alpha_1}{2w^5}+
\frac{\left(2x_3+x_2{w}^{2} \right)\left(4\alpha_2w^4+4{x_3}^2+4x_3x_2w^2+{x_2}^2w^4\right)}{4w^7}
 \right)\varepsilon^2\\\nonumber
&+\left(\gamma_3x_1-\frac{\left(2x_3+x_2w^2\right) \left(
\alpha_2w^2 x_2-2{ \alpha_3}w^2 +2\alpha_2x_3\right)}{2w^5}\right) \varepsilon^3\\\nonumber
&+\left(\gamma_4 x_1-\dfrac{(x_2 w^2+2 x_3)(\alpha_3 w^2 x_2-2\alpha_4 w^2+2 \alpha_3 x_3)}{w^5}\right)\varepsilon^4
\\\nonumber
&+\left(\gamma_5 x_1-\dfrac{(x_2 w^2+2 x_3)(\alpha_4 w^2 x_2-2\alpha_5 w^2+2 \alpha_4 x_3)}{w^5}\right)\varepsilon^5
+\mathcal{O}(\varepsilon^6),\\\nonumber
\dot{x}_2=&wx_1
+\dfrac{\varepsilon\left(2d w^2 x_3- w^2x_2-2 x_3\right) \left(\beta_5
\varepsilon^4+\beta_4\varepsilon^3+\beta_3\varepsilon^2+\beta_2\varepsilon+\beta_1\right)}{w^4}+\mathcal{O}(\varepsilon^6)\\\nonumber
\dot{x}_3=&\left(\frac{\beta_1 x_3}{w^2}+\frac{\beta_1 x_2}{2}-\beta_1dx_3 \right) \varepsilon
+\left(\frac{\beta_2x_3}{w^2}+\frac{\beta_2 x_2}{2}-\beta_2dx_3 \right)\varepsilon^2+ 
\left(\frac {{ \beta_3}x_3}{w^2}+\frac{\beta_3x_2}{2}-{ \beta_3}dx_3\right)\varepsilon^3\\\nonumber
&+\left(\dfrac{\beta_4 x_3}{w^2}+\dfrac{\beta_4 x_2}{2}-\beta_4 d x_3\right)\varepsilon^4
+\left(\dfrac{\beta_5 x_3}{w^2}+\dfrac{\beta_5 x_2}{2}-\beta_5 d x_3\right)\varepsilon^5+\mathcal{O}(\varepsilon^6).
\end{align*}
Now, by implementation of cylindrical coordinate 
\begin{eqnarray*}
(x_1,x_2,x_3)=(r \cos{(\theta)},r \sin{(\theta)},z), \,\, \mbox{ where } \,\, \dot{\theta}=w+\mathcal{O}(\varepsilon)
\end{eqnarray*}
and doing a time rescaling, the FitzHugh-Nagumo system \eqref{FZ} turns into 
\begin{eqnarray*}\label{cylindrical}
\dfrac{dr}{d\theta}&=&
\Bigg(\frac{1}{2}\cos(\theta)(1-\cos^2(\theta))r^2+
\dfrac{r\cos(\theta) \left(r\cos^2(\theta)+2\alpha_1\sin(\theta)+4\sin(\theta)z-r\right)}{2w^2}\\\nonumber
&&+\dfrac{\cos^2(\theta)\gamma_1}{w}r+\dfrac{\beta_1 \left(2 d\sin(\theta) z+r\cos^2(\theta)-r\right)}{w^3}+
\dfrac{2\cos(\theta)z \left(\alpha_1-r\sin(\theta)+z\right)}{w^4}\\\nonumber
&&-\dfrac{2 z\left(\beta_1 w \sin(\theta)+z\cos(\theta)\right)}{w^6}\Bigg)\varepsilon
+\mathcal{O}(\varepsilon^6),\\\nonumber
&=&\varepsilon F^1_1(\theta,r,z)+\sum_{j=2}^5 \varepsilon^i F^1_j(\theta,r,z)+\mathcal{O}(\varepsilon^6),\\\nonumber
\frac{dz}{d\theta}&=&
\dfrac{\beta_1}{w}\left(\frac{1}{w^2}z+\frac{1}{2}\sin(\theta) r- d z \right)\varepsilon
+\mathcal{O}(\varepsilon^2),\\\nonumber
&=&\varepsilon F^2_1(\theta,r,z)+\sum_{j=2}^5 \varepsilon^i F^2_j(\theta,r,z)+\mathcal{O}(\varepsilon^6).
\end{eqnarray*}
Since the expression of $F^i_j(\theta,r,z)$ for $i=1,2$ and $j=2,3,4,5$ are too long, we skip them here.
Now, $g_1(r,z)$ is provided by
\begin{eqnarray}\label{AV1}
g^1_1(r,z)=\frac{r\pi\left(\gamma_1w^2-\beta_1\right)}{w^3}, \,\, 
g^2_1(r,z)=\frac{2\beta_1 z\pi (1-d w^2)}{w^3}.
\end{eqnarray}
\subsection{Existence of periodic solution}\label{B1}
Now, we investigate the conditions for the presence of a periodic solution $\phi(t,x)$ in the FitzHugh-Nagumo system \eqref{FZ} for the Family B.

\begin{proposition}\label{propEcaseB1}
Let $(a,b,c)=\left(-w^2+\alpha(\varepsilon),\beta(\varepsilon),\gamma(\varepsilon)\right)$ where
$\alpha(\varepsilon),$ $\beta(\varepsilon),$ and $\gamma(\varepsilon)$ are given by \eqref{alphabetagamma}. The followings hold.
\begin{itemize}
\item[$\mathsf{B1}$.]
Let
$d=\dfrac{1}{w^2},$ $\beta_1=\gamma_1 w^2,$ 
and $k_0>0,$ which is given by \eqref{valuesB}.
For the $|\varepsilon| \neq 0$ enough small, the FitzHugh-Nagumo system \eqref{FZ} has the periodic solution $\phi(t,\varepsilon)$ given by 
\begin{eqnarray}\label{cycle1}
x(t,\varepsilon)&=&\varepsilon \left(\frac{z}{w^2}+\frac{r \sin(w t)}{2}\right)+\mathcal{O}(\varepsilon^2),\\\nonumber
y(t,\varepsilon)&=&\varepsilon z+\mathcal{O}(\varepsilon^2),\\\nonumber
z(t,\varepsilon)&=&\varepsilon \frac{w r\cos(w t)}{2}+\mathcal{O}(\varepsilon^2),
\end{eqnarray}
such that $\phi(t,\varepsilon)$ converges to $(0,0,0)$ as
$\varepsilon$ converges to $0.$
\item[$\mathsf{B2}$.]
Let $d\neq \dfrac{1}{w^2},$ $\beta_1=\gamma_1 w^2,$ $\beta_2=w^2 \gamma_2 -\alpha_1 \gamma_1$
and $\eta>0,$ which is given by \eqref{valuesB}.
For the sufficiently small $|\varepsilon| \neq 0$, the FitzHugh-Nagumo system \eqref{FZ} has a periodic solution $\phi(t,\varepsilon)$ such that $\phi(t,\varepsilon)$ approaches $(0,0,0)$ as
$\varepsilon$ converges to $0.$
\end{itemize}
\end{proposition}

\begin{proof}
The first-order averaged function \eqref{AV1} only has the trivial zero and 
due to the fact that $(0,0)$ is not within the domain, it is not consistent with a periodic solution.

Consider the conditions given by $\mathsf{B1}.$ Then, $g_1(r,z)=0$ and the second-order averaged function is given by
\begin{align}\label{g2B1}
g_2^1(r,z)&= 
\dfrac{r\pi (2\gamma_1 z-w^2\gamma_1\alpha_1-2 w^2\gamma_1 z-w^2\beta_2+w^4 \gamma_2)}{w^5}
,\\\nonumber
g_2^2(r,z)&=
\dfrac{\gamma_1\pi \left((w^2-1) (r^2 w^4+8 z^2)+8 w^2\alpha_1 z\right)}{4 w^5}.
\end{align}
The second-order averaged function has the non-trivial solution 
\begin{equation*}
\left(r^{+}_{\star}, z_{\star}\right)=\left( 
\dfrac{\sqrt{2}\sqrt{k_0}}{\left|(w^2-1)\gamma_1\right|}, \dfrac{w^2 (w^2\gamma_2-\gamma_1\alpha_1-\beta_2)}{2 \gamma_1 (w^2-1)}\right),
\end{equation*}
where $k_0$ is given by \eqref{valuesB}. 
This solution is real due to the positivity of $k_0.$
For $(w^2-1)\gamma_1>0$ the solution $(r^{+}_{\star},z_{\star})$ is in the domain of $g_2.$ Moreover, 
\begin{eqnarray*}
\mbox{det}\left(D{g}_2\left(r^{\pm}_{\star}, z_{\star}\right)\right)=\dfrac{2\pi^2 k_0}{ w^6}\neq 0,
\end{eqnarray*}
and by using Theorem \ref{thm1} the first statement of the proposition will be proved.

Note that going back through the change of variables
$(x_1,x_2,x_3)\mapsto(r \cos{(\theta)},r \sin{(\theta)},z)$, $(X,Y,Z)\mapsto\varepsilon (x_1,x_2,x_3)$ and $
(x,y,z)\mapsto \left(\dfrac{2Z+w^2 Y}{2w^2}, Z,\dfrac{1}{2}w X\right)$, under the assumptions of Proposition \ref{propEcaseB1}, the differential system \eqref{FZ}  has the limit cycle as \eqref{cycle1}.

Now, we prove the proposition under the condition $\mathsf{B2}$.
We utilize the approach presented in section \eqref{secLya}, combining Lyapunov-Schmidt reduction with averaging theory.
Firstly, let $\beta_1=\gamma_1 w^2$. Then, the first-order averaged function is transformed into
\[
g_1(r,z)=\left(0,\frac{2\gamma_1 z\pi(1-dw^2)}{w}\right).
\]

It is clear that function $g_1$ has a continuum zeros on the
$\mathcal{Z}=\{\mathsf{z}_r=\left(r, \beta(r)\right)=(r,0): r>0\},$ besides we have
\begin{eqnarray}\label{gam0}
D_{g_1}(\mathsf{z}_r)=\left[\begin{array}{cc} 
0&0\\ \noalign{\medskip}
0&\dfrac{2\gamma_1 \pi(1-dw^2)}{w}
\end {array} \right] 
=\left[\begin{array}{cc} 
\Lambda_r &\Gamma_r\\ \noalign{\medskip}
B_r &\Delta_r
\end {array} \right].
\end{eqnarray}
The first-order bifurcation function is given by
\begin{eqnarray*}
f_1(r)=\Gamma_r \zeta_1(r)+\pi g_2(\mathsf{z}_r)
\end{eqnarray*}
where 
\begin{eqnarray*}
\zeta_1(r)=-\Delta^{-1}_r\pi^{\perp} g_2(\mathsf{z}_r)=-\dfrac{r(r-r w^2-4\gamma_1 w+4\gamma_1 w^3 d)}{8(d w^2-1)}.
\end{eqnarray*}
So, $f_1(r)=\pi g_2(\mathsf{z}_r).$ Therefore, we need to calculate the second-order averaged function  which is given by
\begin{eqnarray*}
g_2(r,z)=\frac{1}{2}\int_0^{2\pi} 2F_2(\tau,\mathsf{z})+2\frac{\partial F_1}{\partial\mathsf{x}}(\tau,\mathsf{z})y_1(\tau,\mathsf{z}) d\tau,
\end{eqnarray*}
where $y_1(\tau,\mathsf{z})=g_1(\tau,\mathsf{z})=\int_0^{t} F_1(\tau,\mathsf{z}) d\tau.$ Then
\begin{align*}\label{AV2}
g_2^1(r,z)=&
\frac {\pi  \left( w^{2}-1 \right)  \left( \gamma_1w^{2}+
w^{2}\beta_1d-4\beta_1 \right) z r}{w^7}+\frac {2\pi \beta_1 \left( w^{2}d-1 \right) \left( 2 w^{2}\beta_1d
+\gamma_1w^{2}-3\beta_1 \right) z}{w^8}\\\nonumber
&+\frac {\pi  \left( w^{4} \left( 2\gamma_2w
+{\gamma_1}^{2}\pi  \right) + \left( w \left( \gamma_1\alpha_1
-2\beta_2 \right)  \right) ^{3}-w \left( 2\gamma_1\pi w
+3\alpha_1 \right) \beta_1+\pi{\beta_1}^{2}  \right) r}{2w^6}
,\\\nonumber
g_2^2(r,z)=&\frac{\pi\beta_1 \left(w^2-1 \right) {r}^2}{4w^3}-\frac{\pi\beta_1 \left(2w^2\beta_1d+\gamma_1w^2
-3\beta_1\right) r}{2w^{4}}-\frac{2\pi\beta_1 \left( w^2-1 \right) \left( w^2d-2\right){z}^2}{w^7}\\\nonumber
&+\frac {\pi \left( 2\pi {\beta_1}^2w^{4}{d}^2-2dw^{5}\beta_2+\left(2\beta_2 -d\beta_1\alpha_1\right) w^3-4\pi {\beta_1}^2w^2d+3\beta_1 w{\alpha_1}+2\pi {\beta_1}^2\right) z}{w^{6}}.
\end{align*}
Now, we  can calculate the first-order bifurcation function as
\begin{eqnarray*}
f_1(r)=\frac{r\pi\left(\gamma_2 w^2-\beta_2-\alpha_1 \gamma_1\right)}{w^3}.
\end{eqnarray*}

The function $f_1(r)$ does not have a positive simple zero.
Then, we should impose some constrictions on the parameters in order to vanish $f_1$, then compute the zeros for the second-order bifurcation function. Accordingly, we take $\beta_2=w^2\gamma_2-\alpha_1\gamma_1.$

The second-order bifurcation function given by
\begin{eqnarray*}
f_2(r)= \frac{1}{2}\Gamma_r \zeta_2(r)+\frac{1}{2}\frac{\partial^2 \pi g_1}{\partial b^2}(\mathsf{z}_r) \zeta_1^2(r)
+\frac{\partial \pi g_2}{\partial b}(\mathsf{z}_r) \zeta_1(r)+\pi g_3 (\mathsf{z}_r),
\end{eqnarray*}
where due to $\Gamma_r$ vanishes zero, the calculation of $\zeta_2(r)$ is not needed.

However, 
we have to obtain the third-order averaged function which is given by
\begin{eqnarray*}
g_3(r,z)=\int_0^{2\pi} F_3(\tau,\mathsf{z})+\frac{\partial F_2}{\partial\mathsf{x}}(\tau,\mathsf{z})y_1(\tau,\mathsf{z})+\dfrac{1}{2}\frac{\partial^2 F_1}{\partial\mathsf{x}^2}(\tau,\mathsf{z})y_1^2(\tau,\mathsf{z})+\dfrac{1}{2}\frac{\partial F_1}{\partial\mathsf{x}}(\tau,\mathsf{z})y_2(\tau,\mathsf{z}) d\tau,
\end{eqnarray*}
where for $\mathsf{x}=(r,z),$ $y_1=(y^1_1,y^1_2)$, and
\begin{eqnarray*}
\frac{\partial^2 F}{\partial\mathsf{x}^2}(\tau,\mathsf{z})=
\frac{\partial^2 F}{\partial r^2}(\tau,\mathsf{z}) {y_1^1}^2
+\frac{\partial^2 F}{\partial r \partial z}(\tau,\mathsf{z}) y^1_1 y^1_2
+\frac{\partial^2 F}{\partial z \partial r}(\tau,\mathsf{z})y^1_2 y^1_1
+\frac{\partial^2 F}{\partial z^2}(\tau,\mathsf{z}) {y^1_2}^2.
\end{eqnarray*}
Since the expression of $g_3^1(r,z)$ and $ g_3^2(r,z)$ are too long, we skip them here.
Finally,  the second-order bifurcation function is given by 
\begin{align*}
f_2(r)=&\dfrac{\pi r}{16 w^5 (1-d w^2)}\Big(\gamma_1(1-5 w^2+3 d w^2+w^4-3 d w^4+3 d w^6)r^2\\\nonumber
&\hspace*{2.5cm}+16 w^2(d w^2-1)(\beta_3+\alpha_2\gamma_1+\alpha_1\gamma_2-\gamma_3 w^2-{\gamma_1}^3 d^2 w^4+{\gamma_1}^3 d w^2)\Big),\\\nonumber
=&\dfrac{\pi r}{16 w^5 (1-d w^2)} (r^2-16w^2\eta)
\end{align*}
where $\eta$ is given by \eqref{valuesB}. So, for $\eta>0,$ 
the second bifurcation function, $f_2(r)$, has the unique zero
\begin{eqnarray*}
r^{\ast}=4w\sqrt{\eta},
\end{eqnarray*}
which is the simple and real root due to the positivity of $w$ and $\eta.$
Therefore, the proof is completed from Theorem \ref{thm1}.
\end{proof}
\subsection{Stability of periodic solution}\label{B2}
This part focuses on examining the stability of the periodic solution derived from Proposition \ref{propEcaseB1}.

\begin{proposition}\label{stabilityB1}
The periodic solution provided through Proposition \ref{propEcaseB1} under the conditions $\mathsf{B1}$ is attracting (resp. repelling) if $\beta_2<\gamma_2 w^2$ (resp. $\beta_2>\gamma_2 w^2$).
\end{proposition}

\begin{proof}
The characteristic polynomial for Jacobian matrix of $g_2(\mathsf{z})$  at $(r^{\pm}_{\star},z_{\star})$ given by
\begin{eqnarray*}
p(\lambda)=\lambda^2+\dfrac{2\pi (\beta_2-\gamma_2 w^2)}{w^3}\lambda+\dfrac{2 \pi^2 k_0}{ w^6}.
\end{eqnarray*}
Since $\dfrac{2 \pi^2 k_0}{ w^6}>0,$ according to Theorem \eqref{vehulst} and by using the Routh-Hurwitz criteria we have the following results. The roots of $p(\lambda)$ are on the left hand side of the complex plane when $\beta_2<\gamma_2 w^2$. Besides, for $\beta_2>\gamma_2 w^2,$ 
there is at least one root whose real part is positive
and the periodic solution will be unstable.
\end{proof}

\begin{lemma}\label{CorJordan}
Consider
\[
A(\varepsilon)=  
\begin {pmatrix}
0&  0\\ 
0& a^0_{22}
\end{pmatrix}
+\varepsilon
\begin {pmatrix}
a^1_{11}&  a^1_{12}\\ 
a^1_{21}& a^1_{22}
\end{pmatrix}
+\varepsilon^2
\begin {pmatrix}
a^2_{11}&  a^2_{12}\\ 
a^2_{21}& a^2_{22}
\end{pmatrix}+\mathcal{O}(\varepsilon^3),
\]
where
\begin{eqnarray}\label{conditionn}
a^0_{22}\neq 0.
\end{eqnarray} For sufficiently small $|\varepsilon|\neq 0,$ there is an invertible matrix
\begin{align*}
T(\varepsilon)&=\begin {pmatrix}
1& 0\\
0 &1
\end{pmatrix}
+\varepsilon
\begin{pmatrix}
 0 & \dfrac{a^1_{12}}{a^0_{22}}\\
 -\dfrac{a^1_{21}}{a^0_{22}} &0
\end{pmatrix}
+\varepsilon^2
\begin{pmatrix}
0 & \dfrac{a^2_{12} a^0_{22}+ a^1_{12}(a^1_{11}-a^1_{22})}{({a^0_{22}})^2}\\
-\dfrac{a^2_{21} a^0_{22}+ a^1_{21}(a^1_{11}-a^1_{22})}{({a^0_{22}})^2} & 0
\end{pmatrix}
\end{align*}
such that
\begin{eqnarray*}
T^{-1}A(\varepsilon)T=
\begin{pmatrix}
0&0\\
0& a^0_{22}
\end{pmatrix}
+\varepsilon
\begin{pmatrix}
a^1_{11}&0\\
0& a^1_{22}
\end{pmatrix}
+\varepsilon^2
\begin{pmatrix}
\dfrac{a^0_{22} a^2_{11}-a^1_{12}a^1_{21}}{a^0_{22}}
&0\\
0& \dfrac{a^0_{22} a^2_{22}+a^1_{12}a^1_{21}}{a^0_{22}}
\end{pmatrix}+\mathcal{O}(\varepsilon^4).
\end{eqnarray*}
\end{lemma}

\begin{proposition}\label{stabilityB2}
Consider $\lambda_{1,2}$ as defined in \eqref{valuesB}.
Then,
\begin{itemize}
\item
If $\lambda_{1,2}<0$ (resp. $\lambda_{1,2}>0$) then the periodic solution provided by the Proposition \ref{propEcaseB1} under the conditions $\mathsf{B2}$ is attracting (resp. repelling).
\item
The periodic solution provided by the Proposition \ref{propEcaseB1} under the conditions $\mathsf{B2}$ is unstable whenever $\lambda_1\lambda_2<0.$
\end{itemize}
\end{proposition}
\begin{proof}
Consider the Poincar\'e map
\begin{eqnarray*}
\Pi(\mathbf{z}(\varepsilon),\varepsilon)=
\mathbf{z}+\varepsilon \left(g_1(\mathbf{z})+\varepsilon g_2(\mathbf{z})+\varepsilon^2 g_3(\mathbf{z})\right)+\mathcal{O}(\varepsilon^4),
\end{eqnarray*}
and the Taylor expansion of the Jacobian matrix $D_z\Pi(\mathbf{z}(\varepsilon),\varepsilon)$ around $\varepsilon=0$ as
\begin{eqnarray}\label{DZ}
D_z\Pi(\mathbf{z}(\varepsilon),\varepsilon)=Id+\varepsilon \Gamma_0+\varepsilon^2\Gamma_1+\varepsilon^3 \Gamma_2+\mathcal{O}(\varepsilon^4)=Id+\varepsilon \Gamma(\varepsilon)+\mathcal{O}(\varepsilon^4),
\end{eqnarray}
where $\Gamma_0$ is given by \eqref{gam0} and
condition \eqref{conditionn} is satisfied.
Further,
$\mathbf{z}(\varepsilon)=\mathbf{z}^{\ast}+\varepsilon \mathbf{z}_1+\varepsilon^2 \mathbf{z}_2+\mathcal{O}(\varepsilon^4),$
where
\begin{align*}
\mathbf{z}^{\ast}=&\mathsf{z}_r=\left( r^{\star},0\right), \\\nonumber
\mathbf{z}_1=&\left(-f_3(u^{\ast})\left(Df_2(u^{\ast})\right)^{-1},\zeta_1(u^{\ast})+D\beta(u^{\ast})u_1\right)=\left(-\dfrac{f_3(r^{\star})}{f'_2(r^{\star})},\zeta_1 (r^{\star}) \right),\\\nonumber
\mathbf{z}_2=&\Big(-\frac{1}{2}\left(Df_2(u^{\ast})\right)^{-1}\left(D^2 f_2(u^{\ast})u_1 \odot u_1+2D f_3 (u^{\ast})u_1+2f_4(u^{\ast}) \right),\zeta_2(u^{\ast})+D\zeta_1(u^{\ast})u_1\\\nonumber
&+ D\beta(u^{\ast})u_2+\frac{1}{2}D^2\beta(u^{\ast})u_1\odot u_1 \Big)\\\nonumber
&+\left(\dfrac{f_3(r^{\star})f'_2(r^{\star})}{\left(f'_2(r^{\star})\right)^2}
-\dfrac{f_4(r^{\star})}{\left(f'_2(r^{\star})\right)^2}-\dfrac{f^2_3(r^{\star})f''_2(r^{\star})}{2\left(f'_2(r^{\star})\right)^3},\zeta_2(r^{\star})-\dfrac{f_3(r^{\star})}{f'_2(r^{\star})}\zeta'_1 (r^{\star}) \right).
\end{align*}
Note that for calculation of $f_3$ and $f_4$, four and five-order averaged functions are needed. Duo to the extent of these expressions, we do not mention them here. According to Lemma \eqref{CorJordan} 
an invertible matrix $P(\varepsilon)$  exists
such that the matrix $P^{-1}(\varepsilon) \Gamma(\varepsilon)P(\varepsilon) $ 
is written
as the summation of diagonal matrices
\[
P^{-1}(\varepsilon)  \Gamma(\varepsilon)P(\varepsilon)=\Delta_0+\varepsilon \Delta_1+\varepsilon^2 \Delta_2+\mathcal{O}(\varepsilon^3),
\]
where
\begin{align*}
\Delta_0&=\begin{pmatrix}
0 & 0 \\ 
0 & \dfrac{2\gamma_1 \pi(1-dw^2)}{w}
\end{pmatrix},
\\\nonumber
\Delta_1&=\begin{pmatrix}
0 & 0 \\ 
0 & \dfrac{\pi \bigg(\alpha_1\gamma_1 (1+d w^2)+2 (d w^2-1) w (\pi\gamma_1^2 w^2 d-\gamma_2w-\pi\gamma_1^2)\bigg)}{w^3}
\end{pmatrix},
\end{align*}
and
\begin{align*}
\Delta_2=&\begin{pmatrix}
\dfrac{2\pi \left(w^2\gamma_1^3 d-w^4\gamma_1^3 d^2-w^2\gamma_3+\gamma_1\alpha_2+\alpha_1\gamma_2+\beta_3\right)}{w^3} & 0 \\ 
0 & \delta^2_{1,1} 
\end{pmatrix}.
\end{align*}
Due to the complexity of the expressions in the $\delta^2_{1,1}$, we omit it here.
Further, since $ \Gamma(\varepsilon)$ and $P^{-1}(\varepsilon)  \Gamma(\varepsilon)P(\varepsilon)$ are similar for $\varepsilon>0$ enough small, Theorem \eqref{Murdock} concludes that they have the approximately equal eigenvalues with error $\mathcal{O}(\varepsilon^4)$,
and the eigenvalues of \eqref{DZ} are written as
\begin{eqnarray*}
\lambda_1(\varepsilon)=1+\varepsilon\lambda_1+\mathcal{O}(\varepsilon^2), \,\, \lambda_2(\varepsilon)=1+\varepsilon^3\lambda_2+\mathcal{O}(\varepsilon^4),
\end{eqnarray*}
where $\lambda_1$ and $\lambda_2$ are given by  \eqref{valuesB}. So, the statements are proved straightforwardly.
\end{proof}

\subsection{Bifurcation of an invariant torus}\label{B3}
Consider the hypothesis of Proposition \ref{propEcaseB1}($\mathsf{B1}$). Then, $g_1\equiv 0,$ $g_2$ is given by \eqref{g2B1}, and $g_3=(g_3^1,g_3^2)$ is given by
\begin{align*}\nonumber
g_3^1(r,z)=&\frac {\pi { \gamma_1} ( -5w^{4}+7w^{2}-5) {r}^{3}}{16w^{5}}-{\frac {3\pi { \gamma_1}
( 3w^{4}-5w^{2}+3 ) {z}^{2}r}{w^{9}}}-\frac{\pi {\gamma_1}^2 (w^2-1)}{w^4}r^2\\\nonumber
&-{\frac {\pi ( 6{ \beta_2}w^{2}+14w^{2}{\gamma_1}{ \alpha_1}-18{ \gamma_1}{ \alpha_1}
-6{ \beta_2}+2\gamma_2 w^4+2\gamma_2 w^2) zr}{w^{7}}}
,\\\nonumber
&-\dfrac{\pi w^{2} ( 2{\gamma_3}w^{4}-2w^{2}{\gamma_1}{ \alpha_2}-2w^{2}{ \beta_3}
-3{ \gamma_1}{{ \alpha_1}}^{2}-3{ \alpha_1}{ \beta_2}+w^2\alpha_1 \gamma_2 ) r}{w^{7}}
,\\\nonumber
g_3^2(r,z)=&
\frac {\pi { \gamma_1} (7w^{4}-8w^{2}+7 ) {r}^{2}z}{8w^{5}}+\frac {\pi(3 w^{2}{ \gamma_1}{ \alpha_1}+4{ \beta_2}w^{2}-7{ \gamma_1}{ \alpha_1}-4{ 
\beta_2} ) {r}^{2}}{16w^{3}}+{\frac {2\pi {{ \gamma_1}}^{2}
 ( w^{2}-1 ) rz}{w^{4}}}\\\nonumber
&+\frac {\pi {\gamma_1}(2{\gamma_1}{ \alpha_1}-w^{2}{ \gamma_2}
+{ \beta_2}) r}{2w^{2}}+{\frac {\pi  ( 7w
^{2}{ \gamma_1}{ \alpha_1}-9{ \gamma_1}{ \alpha_1}+2{ 
\beta_2}w^{2}-2{ \beta_2} ) {z}^{2}}{w^{7}}}\\\label{labelg3}
&+\frac {2\pi { \gamma_1} ( 3w^{4
}-5w^{2}+3 ) {z}^{3}}{w^9}+{\frac {\pi(2w^{2}{ \gamma_1}{ \alpha_2}+3{ \gamma_1}{{\alpha_1}}^{2}+2{ \alpha_1}{ \beta_2} ) z}{w^{5}}}.
\end{align*}

\begin{proposition}\label{TorusB}
Let system \eqref{FZ} satisfy the hypothesis of Proposition \ref{propEcaseB1}($\mathsf{B1}$). If $(w^2-1) \gamma_1 \neq 0$, then there is a smooth curve $\beta(\varepsilon)=\beta^{\star}+\varepsilon \beta1+\varepsilon^2 \beta2+\mathcal{O}(\varepsilon^3)$ where $\beta^{\star}=\gamma_2 w^2,$ and interval $J_{\varepsilon}$  containing $\beta(\varepsilon)$ 
so as the parameter $\beta_2$ crosses $\beta(\varepsilon,)$ only one invariant torus emerges from the periodic solution $\phi(t,\beta(\varepsilon),\varepsilon).$
Whenever $\beta_2 \in J_{\varepsilon}$ and $l_{13} (\beta_2-\beta(\varepsilon)) >0,$ such this torus exists while surrounds the periodic solution. Moreover, if $l_{1,3}<0$ (resp. $l_{1,3}>0$) the torus is asymptotically stable (resp. unstable) surrounding the unstable (resp. stable) periodic orbit.
\end{proposition}

\begin{proof}
The arguments of the proof are very similar to those used in the Theorem \ref{Torus}.
Assume $\beta^{\star}=\gamma_2 w^2$ and $p(\lambda)$ denotes the characteristic polynomial of the Jacobian matrix $\dfrac{\partial g_2}{\partial \mathbf{z}}(\mathbf{z}_{\beta_2},\beta_2).$ The roots of $p(\lambda)$ is given by $\lambda(\beta_2)=\alpha(\beta_2)\pm i \eta(\beta_2),$ where
\begin{eqnarray*}\label{H1H2}
\alpha(\beta^{\ast})=0,\,\, 
\eta(\beta^{\ast})=\dfrac{\sqrt{2}\pi |\alpha_1\gamma_1| }{ w^3}=w_0, \,\, d_0=\dfrac{d\alpha(\beta_2)}{d\beta_2}\bigg|_{\beta_2=\beta^{\star}}=-\pi < 0.
\end{eqnarray*}
So, the hypotheses $\textbf{C1}$ and $\textbf{C2}$ are verified. Now, by using the implicit function theorem and these hypotheses, there are functions $\mathsf{z}_1$ and $\beta1$ such that
\begin{eqnarray*}
\zeta(\beta_2,\varepsilon)=\mathbf{z}_{\beta_2} +\varepsilon \mathsf{z}_1+\mathcal{O}(\varepsilon^2) \,\, \mbox{and} \,\, \mu(\varepsilon)=\beta^{\ast}+\varepsilon \beta1+\mathcal{O}(\varepsilon^2).
\end{eqnarray*}
We ignore these functions here due to cumbersome expressions.

Consider the change of variable 
$\mathbf{z}=\mathbf{x}+\zeta(\mu,\varepsilon),$  the change of parameter $\beta_2=\sigma+\gamma_2 w^2,$ and the linear change of variable 
\begin{eqnarray*}
\mathsf{x}=\mathsf{V} \mathsf{y}=\begin{pmatrix}
\dfrac{1}{2} & 
-\dfrac{\sqrt{2} \left(4\sqrt{2}\gamma_1 w^5+5 \alpha_1 w^4-8\sqrt{2}\gamma_1 w^3-7  \alpha_1 w^2+4\sqrt{2}
\gamma_1 w+5\alpha_1\right)}{16 w^2 (w^2-1)^2}\varepsilon
\\ 
0 & \dfrac{1}{4} w^2 
\end{pmatrix}\mathsf{y}.
\end{eqnarray*}
Then, we obtain the map
\begin{eqnarray*}
 H_{\varepsilon}(\mathbf{y},\sigma)=\mathsf{V}^{-1} \Pi\left( \mathsf{V}\mathbf{y}+\zeta(\mu(\varepsilon)+\beta_2,\varepsilon),\mu(\varepsilon)+\beta_2,\varepsilon\right),
\end{eqnarray*}
where
\begin{eqnarray*}
\Pi(\mathbf{z},\mu,\varepsilon)=\mathbf{z}+\sum_{i=1}^{3} \varepsilon^i g_i(\mathbf{z},\mu)+\mathcal{O}(\varepsilon^{4}).
\end{eqnarray*}
Now, the Taylor expansion of 
\begin{eqnarray*}
 D_{\mathbf{y}} H_{\varepsilon}(0,0)=Id+\varepsilon\mathsf{A}_{1}+\varepsilon^2 \mathsf{A}_{2}+\varepsilon^3\mathsf{A}_{3},
\end{eqnarray*}
in real Jordan form is given by
\begin{align*}
\begin{pmatrix}
1 & 0 \\ 
0 & 1
\end{pmatrix} +
\varepsilon^2
\begin{pmatrix}
0 & \dfrac{\sqrt{2}\pi\alpha_1 \gamma_1}{w^3} \\ 
-\dfrac{\sqrt{2}\pi\alpha_1 \gamma_1}{w^3} & 0
\end{pmatrix}.
\end{align*}
Therefore, we have the map
\begin{equation*}
\mathsf{y}\mapsto H_{\varepsilon}(\mathsf{y},\sigma)=\left( H^1_{\varepsilon}(\mathsf{y},\sigma),H^2_{\varepsilon}(\mathsf{y},\sigma) \right).
\end{equation*}
In the end, we calculate  the Lyapunov coefficient \eqref{FormulaLyapunov} where $\mathbf{p}=\left(\dfrac{1}{2},-\dfrac{i}{\sqrt{2}}\right),$ $e^{i \theta_{\varepsilon}}=1+\varepsilon (i w_0)+\mathcal{O}(\varepsilon^2),$ and
\begin{align*}
B_{\varepsilon}(u,v)&={\scalebox{0.9}{$\left(\dfrac{\gamma_1 \pi \mbox{sign}(\alpha_1\gamma_1) (w^2-1) (u_1 v_2+u_2v_1)}{2w^3} \varepsilon^2, -\dfrac{\gamma_1 \mbox{sign}(\alpha_1\gamma_1)\pi (w^2-1) (2u_2v_2+u_1 v_1)}
{2w^3}\varepsilon^2\right)+\mathcal{O}(\varepsilon^3),$}}\\\nonumber
C_{\varepsilon}(u,v,w)&=\bigg(
\frac{3\pi \gamma_1 \left( 4 \left( 5w^2-3w^4-3 \right)  \left( u_2v_2w_1+u_2v_1w_2+u_1v_2w_2 \right) +u_1v_1w_1 \left( 5w^4-7w^2+5 \right)  \right) {\varepsilon }^3}{32w^5},\\\nonumber
& \,\,\frac {\pi \gamma_1 \left(\left( 5-7w^{2}+5w^{4}\right)  \left( w_1v_1u_2+u_1w_1v_2
+u_1v_1w_2 \right) +6 \left( 3w^{4}-5w^{2}+3 \right) u_2v_2w_2\right) }{8w^5}\varepsilon^3
\bigg)\\\nonumber
& \,\,+\mathcal{O}(\varepsilon^4).
\end{align*}
Now, by expanding $l^1_{\varepsilon}$ around $\varepsilon=0,$ we have
\begin{equation*}
l_1^{\varepsilon}=\varepsilon l_{1,1}+\varepsilon^2 l_{1,2}+\varepsilon^3 l_{1,3}
+\mathcal{O}(\varepsilon^4),
\end{equation*}
where 
\[
l_{1,1}=l_{1,2}=0, \,\, \mbox{ and } \,\, l_{1,3}=\dfrac{ (41w^4-67w^2+41) \pi\gamma_1 }{128 w^5}.
\]
The existence and stability of this torus will be confirmed by using the Theorem \ref{thmtorus}.
\end{proof}

\section{Numerical Simulations}\label{sec:ex}

In Theorems \ref{teoA} and \ref{teoB}, we provided sufficient conditions for the existence of invariant tori within the parametric family of the FitzHugh-Nagumo system \ref{FZ}. Here, we present numerical examples for each result mentioned in these theorems. It is important to highlight that all coefficients used in this section were carefully chosen based on the criteria specified in Theorems \ref{teoA} and \ref{teoB}. The invariant sets provided here may be difficult to detect using numerical tools alone, as they occur in a very narrow region of the parameter space. In this section, we utilize the Backward Differentiation Formula (BDF) to solve the stiff ordinary differential equations arising from our study. The BDF method, known for its stability properties, is particularly well-suited for handling the rapid changes in the solution that characterize stiff ODEs. We implement a higher-order BDF method.

\subsection{Example 1}
Consider system \eqref{FZ} with the coefficients \eqref{abcA} that satisfy the relationships defined in Theorem \ref{teoA} of family A. The following coefficients correspond to the existence of an invariant torus in the system:
\begin{align*}
a &= -\dfrac{8}{19} + 128 \sqrt{2}\, \epsilon, \,\, & b &= \frac{1}{152}\sqrt{\dfrac{3869}{19}}, \,\, & c &= \frac{1}{64}\sqrt{\dfrac{3869}{19}} +\dfrac{671757 }{2007040000}\,\epsilon, \\
d &= \dfrac{17}{16}, \,\, & \omega &= \dfrac{39}{64}, \,\, & \epsilon &= \dfrac{1}{5000}.
\end{align*}
In this case, the second term of the expansion in epsilon of the Lyapunov coefficient is given by
$$
\ell_{12} = \dfrac{34359738368 \pi  \left(371234123 \sqrt{73511}+3835076608 \pi \right)}{6718817122378946014833}.
$$
In Figure \ref{torA}, the invariant tori in system \eqref{FZ} and its corresponding invariant circle in the Poincar\'e section of the invariant torus are shown.

\begin{figure}[h!]
\begin{minipage}{0.45\textwidth}
    \centering
    \includegraphics[width=6cm]{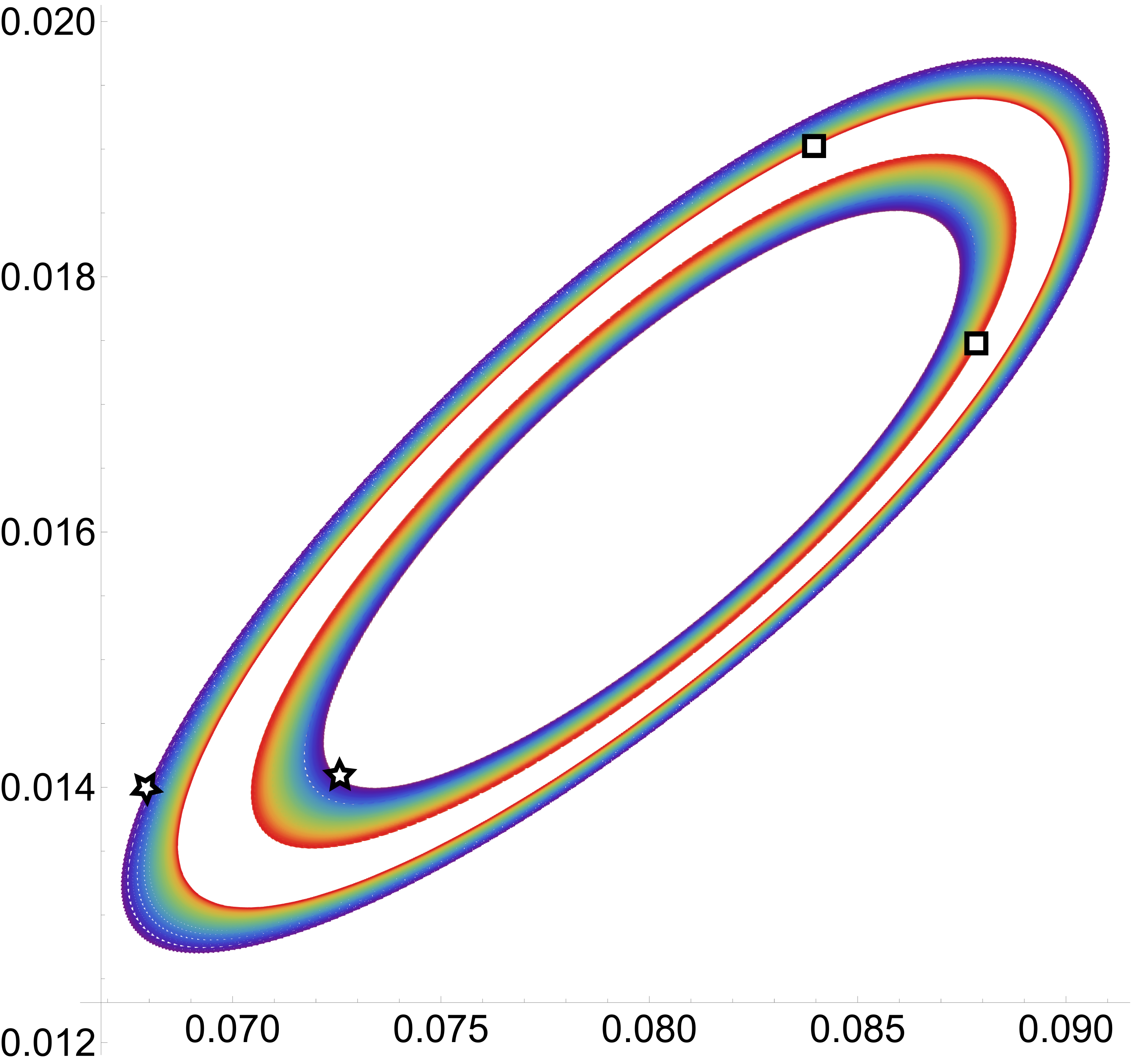}
\end{minipage}
\begin{minipage}{0.45\textwidth}
    \centering
    \includegraphics[width=8cm]{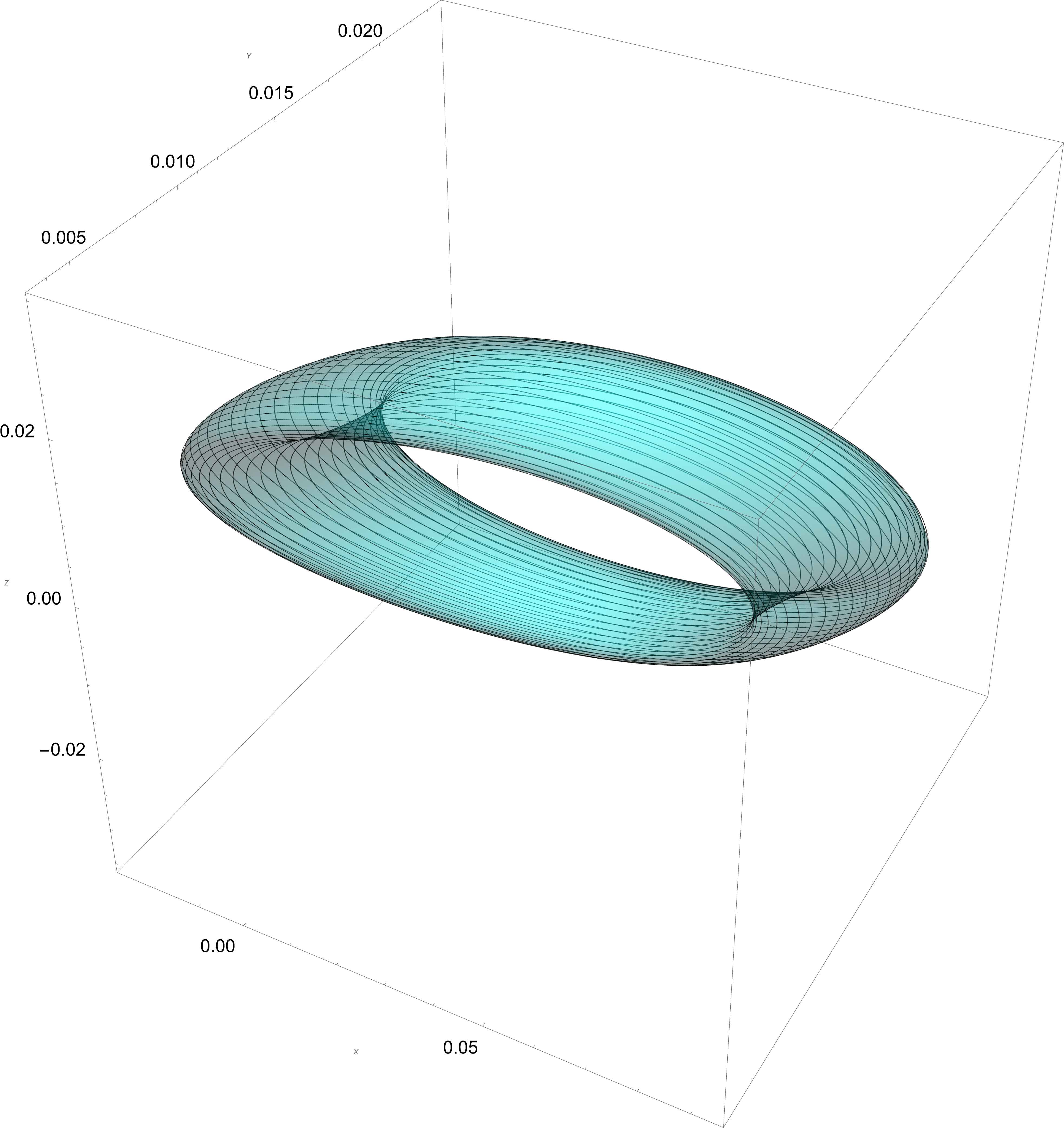}
\end{minipage}
\caption{Left: Solutions starting at $(x_1, y_1,z_1) = (0.06771653, 0.01286177, 0)$ and $(x_2, y_2,z_2) = (0.0724458254, 0.014088585, 0)$ spiraling around the closed invariant curve at the Poincaré section $z=0$ of system \eqref{FZ} with coefficients \eqref{abcA}.The starting points are represented by the star symbol and the endpoints by the square symbol. Right: This figure shows, in solid line, the solution starting at point $p = (0.0692714, 0.0133573, 0)$ entering the vicinity of the invariant torus, represented by the smooth surface, for the system with coefficients $a = -0.3848488$, $b = 0.09388127$, $c = 0.2229681$, and $d = 0.6093750$.  \label{torA}} 
\end{figure}

\subsection{Example 2}
Consider system \eqref{FZ} with the coefficients \eqref{abcB} that satisfy the relationships defined in case B1 of family B. The following coefficients correspond to the existence of an invariant torus in the system:
\begin{align*}
a &= \frac{1}{4} + \frac{1}{2} \epsilon - \epsilon^2 + \dfrac{3}{100} \epsilon^3, \,\, & b &= \dfrac{1}{4} \epsilon - 1.000435384 \epsilon^2 + \dfrac{1}{50} \epsilon^3, \,\, & c &= \epsilon - 4 \epsilon^2 + \dfrac{1}{100} \epsilon^3, \\
d &= 4, \,\, & \omega& = \dfrac{1}{2}, \,\, & \epsilon &= \frac{1}{20}.
\end{align*}

In this case, the third term of the expansion in epsilon of the Lyapunov coefficient is given by 
$$
\ell_{1,3} = -\frac{10725}{8} \pi,
$$
indicating that the system undergoes a Neimark-Sacker bifurcation. Figure \ref{torB} shows the Poincaré section of the invariant torus; here we use the $(r,z)$ cylindrical coordinates to localize the invariant torus. We also show the torus in the original coordinates of \eqref{FZ}.

\begin{figure}[h!]
\begin{minipage}{0.45\textwidth}
    \centering
    \includegraphics[width=6cm]{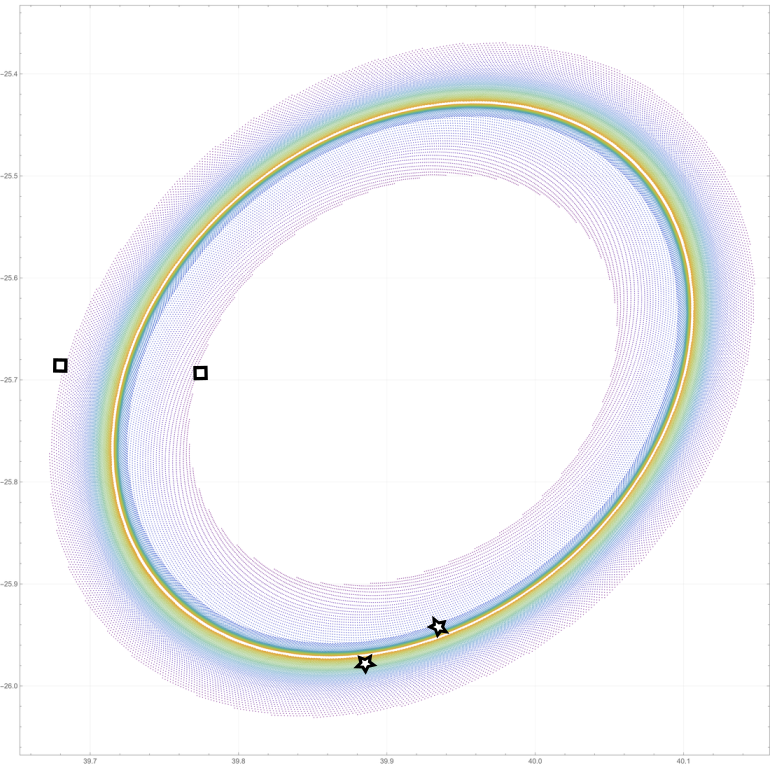}
\end{minipage}
\begin{minipage}{0.45\textwidth}
    \centering
    \includegraphics[width=8cm]{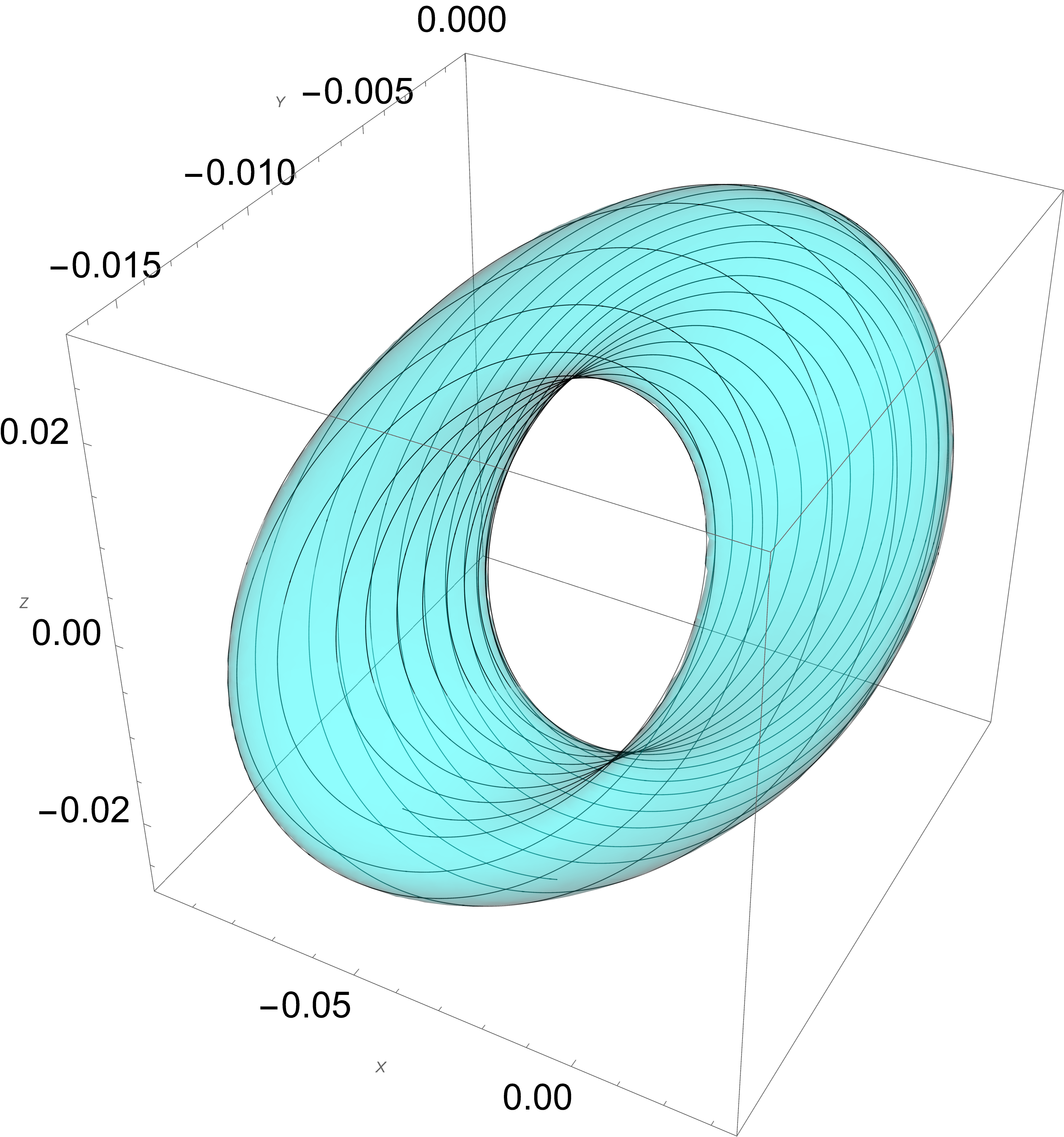}
\end{minipage}
\caption{Left: Solutions starting at $(r_1, z_1) = (39.9008, -25.5074)$ and $(r_2, z_2) = (39.8987, -25.3815)$ spiraling around the closed invariant curve at the Poincaré section of the system in cylindrical coordinates. The starting points are represented by the star symbol and the endpoints by the square symbol. Right: This figure depicts, in solid line, the solution starting at point $p = (-0.05039858, -0.009134, -0.0134404)$ leaving the vicinity of the invariant torus, represented by the smooth surface, for the system with coefficients $a = 0.3001$, $b = -0.05010014$, $c = -0.2004004$, and $d = 4$. \label{torB}}
\end{figure}

\subsection{Example 3}
Consider system \eqref{FZ} with the coefficients \eqref{abcB} that satisfy the relationships defined in family B, satisfying case B2 of Theorem \ref{teoB}. The following coefficients satisfy the conditions for the existence of a small periodic solution:
\begin{align*}
a &= \dfrac{121}{49} + \frac{15 }{7}\e + 2 \e^2 + \e^3, \,\, && b = \dfrac{1452 }{343}\e - \frac{1502 }{343}\e^2 + \e^3, \,\, && c = \dfrac{12 }{7}\e - \frac{2 }{7}\e^2 + \e^3, \\
d &= \dfrac{10}{7}, \,\, && \omega = \dfrac{11}{7}, \,\, && \e = \dfrac{1}{15}.
\end{align*}

Under these conditions, the small periodic solution (see Figure \ref{fop}) that bifurcates from the origin is attracting since
$$
\lambda_1 = -\dfrac{20808 \pi }{3773} \,\, \text{and} \,\, \lambda_2 = -\dfrac{424228304 \pi }{14235529}.
$$

\begin{figure}[h!]
\begin{minipage}{0.45\textwidth}
    \centering
    \includegraphics[width=6cm]{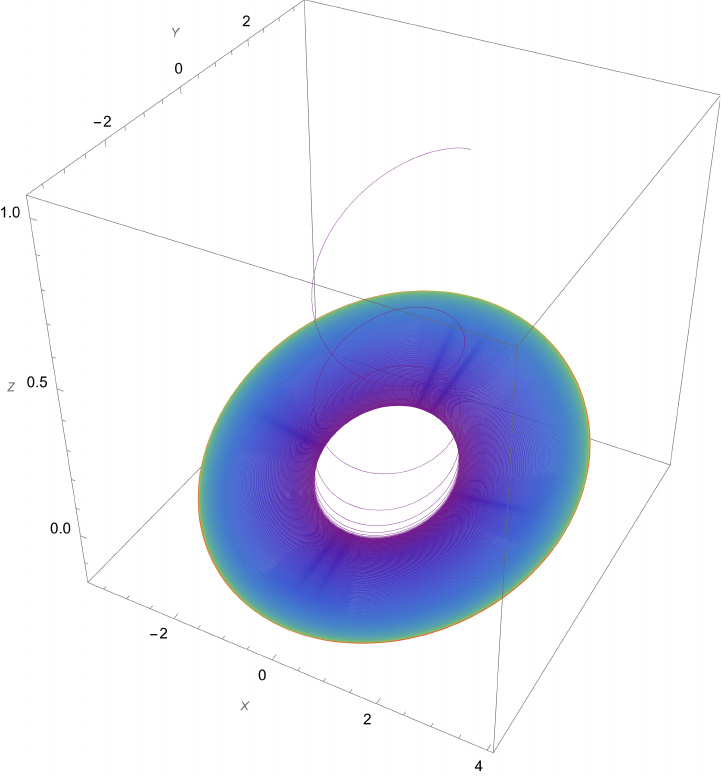}
\end{minipage}
\begin{minipage}{0.45\textwidth}
    \centering
    \includegraphics[width=8cm]{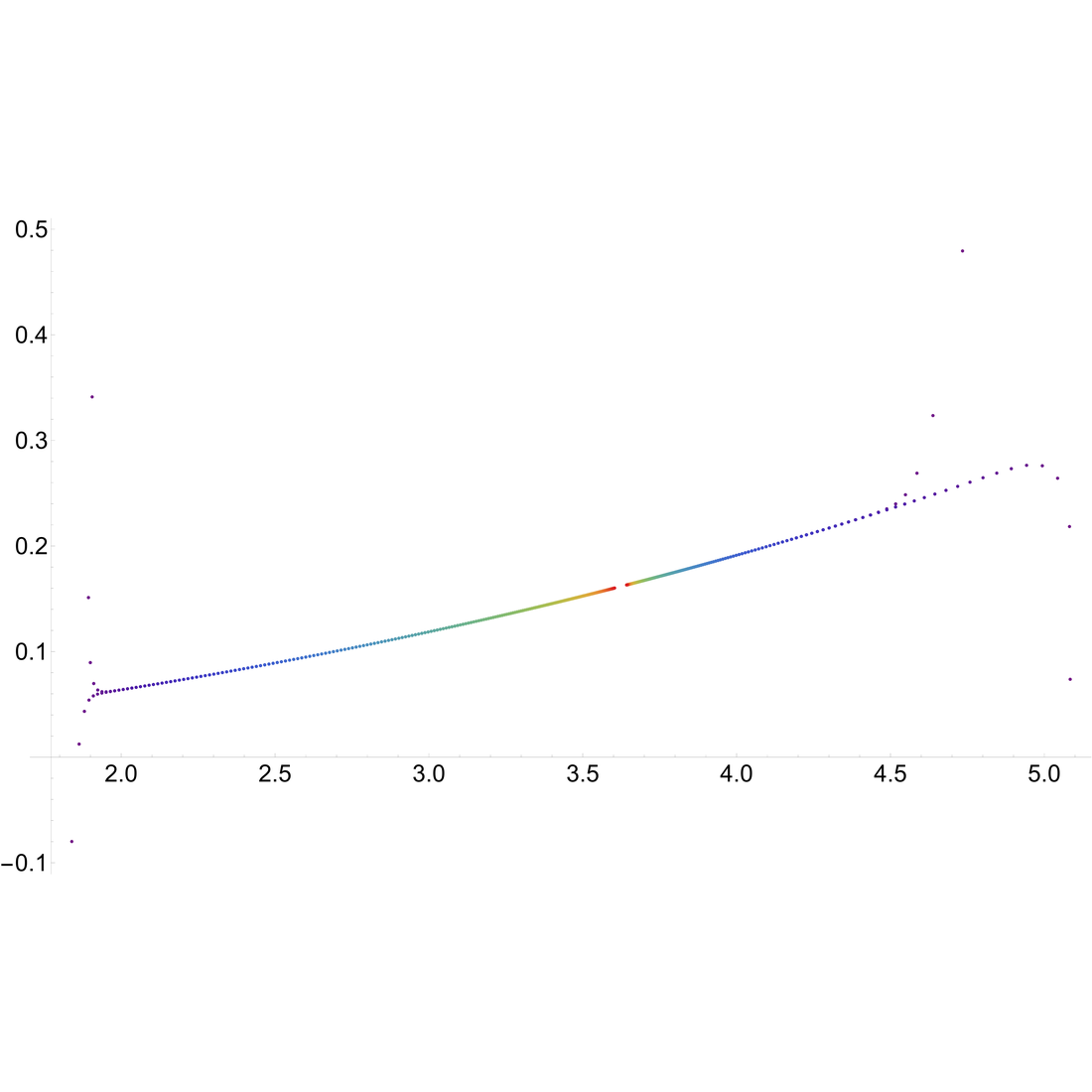}
\end{minipage}
\caption{Left: The solution starting at $p = (1, 1, 1)$ being attracted by the nearby stable periodic solution. Right: The Poincaré section $y = 0$ of the periodic solution. We can see four nearby solutions being attracted to the corresponding fixed point.\label{fop}}
\end{figure}

\section{Conclusion}
In this study, we have thoroughly examined the FitzHugh-Nagumo system, a model for nerve axon excitation and neural signal propagation, focusing particularly on the scenarios where the system exhibits a Hopf-zero equilibrium at the origin. By considering two distinct parameter families, {\bf Family A} and {\bf Family B}, we aimed to extend previously established results on periodic solutions by exploring the existence and stability of invariant tori arising from periodic solutions.

For {\bf Family A}, characterized by a first set of parameter relationships that yield a pair of purely imaginary conjugate eigenvalues and a zero eigenvalue, we re-prove a generic condition on the parameters providing the birth of a periodic solution as a parameter $\e$ crosses zero. We demonstrated the convergence of this solution to the origin and analyzed its stability, showing that the nature of this periodic solution (attracting or repelling) depends on a parameter relationships. Moreover, we identified a bifurcation of an invariant torus from the periodic solution, providing a deeper understanding of the dynamical behavior as the system parameters evolve.

{\bf Family B}, characterized by a second set of parameter relationships that yield a pair of purely imaginary conjugate eigenvalues and a zero eigenvalue, introduced a more intricate scenario, which was divided into two classes, namely, {\bf Class B1} and {\bf Class B2}. For {\bf Class B1}, we revisited and confirmed the existence of a periodic solution previously identified. In addition, its stability characteristics was characterized in terms of relationships on the parameters. Furthermore, we uncovered conditions under which this periodic solution undergoes a torus bifurcation, leading to the emergence of invariant tori. In {\bf Class B2}, we tackled the challenge posed by a continuum of zeros in the first-order averaging function, a scenario not previously investigated. Using advanced averaging theory combined with the Lyapunov-Schmidt reduction, we provided a framework for understanding the bifurcation of periodic solutions in this class. This approach allowed us to delineate the conditions under which these periodic solutions are stable or unstable, shedding light on the complex dynamics that can arise in the FitzHugh-Nagumo system.

\section{Acknowledgments}

MRC is supported by S\~{a}o Paulo Research Foundation (FAPESP) grant 2023/06076-0. DDN is supported by S\~{a}o Paulo Research Foundation (FAPESP) grants 2018/13481-0, 2019/10269-3, and 2022/09633-5, and by Conselho Nacional de Desenvolvimento Cient\'{i}fico e Tecnol\'{o}gico (CNPq) grant 309110/2021-1. NS is supported by S\~{a}o Paulo Research Foundation (FAPESP) grant 2023/04851-7 and by Institute for Research in Fundamental Sciences (IPM) grant 1403370046. NS also thanks Dr. Marco Antonio Teixeira for his support and useful comments.


\begin{thebibliography}{00}
\bibitem{buica07} A. Buic\u{a}, J.P. Fran\c{c}oise, and J. Llibre, 
{Periodic solutions of nonlinear periodic differential systems with a small parameter,}
{\it Commun. Pure Appl. Anal.,} {\bf 6} (2007), 103--111.

\bibitem{NovaesN} M.R C\^andido, D.D Novaes, C. Valls, 
{ Periodic solutions and invariant torus in the Rossler system,}
{\it Nonlinearity,} {\bf 33} (2020), 4512--4538.

\bibitem{NovaesJ} M.R C\^andido, D.D Novaes, 
{ On the torus bifurcation in averaging theory,}
{\it J. Differential Equations,} {\bf 8} (2020), 4555--4576.

\bibitem{Lyapunov}
M.R C\^andido, J. Llibre, D.D Novaes,
Persistence of periodic solutions for higher order perturbed differential systems via Lyapunov-Schmidt reduction, \textit{Nonlinearity,} {\bf 30} (2017), 3560--3586.

\bibitem{RodrigoLlibreb} R.D  Euz\'ebio, J. Llibre, C. Vidal, 
{Zero-Hopf bifurcation in the FitzHugh-Nagumo system,}
{\it Mathematical Methods in the Applied Science,} {\bf 38} (2014), 4289--4299.

\bibitem{FitzHugh}
R. Fitzhugh,
Impulses and physiological state in theoretical models of nerve membrane,
\textit{Biophysical Journal,} {\bf 1} (1961), 445--467.

\bibitem{LlibreNovaes}
J. Llibre, D.D Novaes, 
Improving the averaging theory for computing periodic solutions of the differential equations, \textit{ Z. Angew. Math. Phys.,} {\bf 66} (2015), 1401--1412.

\bibitem{periodicFN}
J. Llibre, C. Vidal, 
Periodic Solutions of a Periodic FitzHugh-Nagumo System, 
\textit{International Journal of Bifurcation and Chaos,} {\bf 13} (2015), 1550180 (6 pages).

\bibitem{NovaesBrouwer}
J. Llibre, D.D Novaes, M.A Teixeira, 
{\it Higher order averaging theory for finding periodic solutions via Brouwer degree,}
Nonlinearity, {\bf 27} (2014), 563--583.

\bibitem{Murdock}
J. Murdock, \textit{Normal forms and unfoldings for local dynamical systems,} (2006), (Berlin: Springer).

\bibitem{Nagumo}
J. Nagumo, S. Arimoto, S. Yoshizawa,
An active pulse transmission line simulating nerve axon,
\textit{ Proceedings of the IRE,} {\bf 50} (1963), 2061--2070.

\bibitem{verhust}
F. Verhulst, \textit{Nonlinear Differential Equations and Dynamical Systems,} (2006), (Berlin:  Springer).
\end{thebibliography}
\end{document}